\documentclass[12pt]{amsart}
\usepackage{amsmath,centernot}
\usepackage{amssymb}
\usepackage[hidelinks]{hyperref}

\newtheorem{theorem}{Theorem}
\newtheorem{lemma}[theorem]{Lemma}
\newtheorem{proposition}[theorem]{Proposition}

\newtheorem*{claim}{Claim}
\numberwithin{theorem}{section}
\numberwithin{equation}{section}

\theoremstyle{definition}

\newtheorem{definition}[theorem]{Definition}
\newtheorem{observation}[theorem]{Observation}
\newtheorem{question}[theorem]{Question}
\newtheorem{problem}[theorem]{Problem}
\newtheorem{remark}[theorem]{Remark}

\newcommand{\bn}{\mathbf{n}}
\newcommand{\cnot}{\centernot}
\newcommand{\mb}{\mathbf}

\date{\today}


\begin{document}

\title[Single recurrence]{Single recurrence in abelian groups}

\author{John T. Griesmer}
\address{Department of Applied Mathematics and Statistics\\ Colorado School of Mines, Golden, Colorado}
\email{jtgriesmer@gmail.com}

\begin{abstract} We collect problems on recurrence for measure preserving and topological actions of a countable abelian group, considering combinatorial versions of these problems as well. We solve one of these problems by constructing, in $G_{2}:=\bigoplus_{n=1}^{\infty} \mathbb Z/2\mathbb Z$, a set $S$ such that every translate of $S$ is a set of topological recurrence, while $S$ is not a set of measurable recurrence.  This construction answers negatively a variant of the following question asked by several authors: if $A\subset \mathbb Z$ has positive upper Banach density, must $A-A$ contain a Bohr neighborhood of some $n\in \mathbb Z$?

We also solve a variant of a problem posed by the author by constructing, for all $\varepsilon>0$, sets $S, A\subseteq G_{2}$ such that every translate of $S$ is a set of topological recurrence, $d^{*}(A)>1-\varepsilon$, and the sumset $S+A$ is not piecewise syndetic.  Here $d^{*}$ denotes upper Banach density.    \end{abstract}

\maketitle

\section{Introduction}

\subsection{Recurrence in dynamics}

 A set $S\subseteq \mathbb Z$ is a \emph{set of measurable recurrence} if for every measure preserving transformation $T:X\to X$ of a probability measure space $(X,\mu)$  and every $D\subseteq X$ having $\mu(D)>0$, there exists $n\in S$ such that $\mu(D\cap T^{-n}D)>0$.  We say that $S$ is a \emph{set of topological recurrence} if for every minimal topological system $(X,T)$, where $X$ is a compact metric space, and $U\subseteq X$ is a nonempty open set, there exists $n\in S$ such that $U\cap T^{-n}U \neq \varnothing$.  Every set of measurable recurrence is also a set of topological recurrence, since every minimal topological system admits an invariant probability measure of full support.  The concepts of measurable and topological recurrence generalize to actions of abelain groups; see Section \ref{sectionDefinitions} for definitions.

Vitaly Bergelson asked if there is a set $S\subseteq \mathbb Z$ which is a set of topological recurrence but not a set of measurable recurrence, and in \cite{Kriz} Igor K\v{r}\'{i}\v{z} constructed such a set.  Alan Forrest, in his doctoral thesis \cite{ForrestThesis}, produced an analogous example for actions of $G_{2}:= \bigoplus_{n=1}^{\infty} \mathbb Z/2\mathbb Z$.  Our main result, Theorem \ref{theoremMainDynamic}, is a more robust example: we find a set $S\subseteq G_{2}$ such that every translate of $S$ is a set of topological recurrence while $S$ is not a set of measurable recurrence.

If $G$ is an abelian group and $A, B\subseteq G$, $g\in G$, we write $A+B$ for the \emph{sumset} $\{a+b:a\in A, b\in B\}$, $A-B$ for the \emph{difference set} $\{a-b: a\in A, b\in B\}$, and $A+g$ for the \emph{translate} $\{a+g: a\in A\}$.

\begin{theorem}\label{theoremMainDynamic}
  There is a set $S \subseteq G_{2}$ such that for all $g\in G_{2}$, $S+g$ is a set of topological recurrence, while $S$ is not a set of measurable recurrence.
\end{theorem}

If $(X,\mu,T)$ is a measure preserving $G$-system and $D\subseteq X$ has $\mu(D)>0$, the set $\operatorname{Ret}_{T}(D):=\{g\in G: \mu(D\cap T^{g}D)>0\}$ is of interest.  One way to study such sets is to identify sets of measurable recurrence.  Part (iii) of Lemma \ref{lemmaDifferenceAndRecurrence} and Lemma \ref{lemmaMeasureCorrespondence} show that Theorem \ref{theoremMainDynamic} is equivalent to the following statement: there is an ergodic measure preserving $G_{2}$-system $(X,\mu,T)$ and $D\subseteq X$ with $\mu(D)>0$ such that $\operatorname{Ret}_{T}(D)$ does not contain a set of the form $g+(B-B)$ where $B\subseteq G_{2}$ is piecewise syndetic and $g\in G_{2}$.  See Section \ref{sectionCombinatorial} for the definition of ``piecewise syndetic."

\subsection{Difference sets}  If $G$ is a countable abelian group and $A\subseteq G$, let $d^{*}(A)$ denote the upper Banach density of $A$; see Sections \ref{sectionBohr} and \ref{sectionUBD} for definitions of Bohr neighborhoods and upper Banach density.  A theorem of F{\o}lner \cite{Folner54} states that when $d^{*}(A)>0$, $A-A$ contains a set of the form $U\setminus Z$, where $U$ is a Bohr neighborhood of $0\in G$ and $d^{*}(Z)=0$.  For $G=\mathbb Z$, K\v{r}\'{i}\v{z}'s construction \cite{Kriz} exhibits a set $A$ having $d^{*}(A)>0$ such that $A-A$ does not contain a Bohr neighborhood of $0$, and in fact $A-A$ does not contain a set of the form $S-S$, where $S\subseteq \mathbb Z$ is piecewise syndetic.  Several authors have asked (\cite{BergelsonRuzsa, RuzsaBook, GriesmerIsr, HegyvariRuzsa}) whether $A-A$ must contain a Bohr neighborhood of \emph{some} $n\in \mathbb Z$ whenever $d^{*}(A)>0$, and this question remains open.   For $G_{p}:= \bigoplus_{n=1}^{\infty} \mathbb Z/p\mathbb Z$, where $p$ is prime, the author constructed in \cite{GriesmerBohr}  sets $A\subseteq G_{p}$ having $d^{*}(A)>0$ such that $A-A$ does not contain a Bohr neighborhood of \emph{any} $g\in G_{p}$.

 From a combinatorial perspective, the sets $\operatorname{Ret}_{T}(D)$ are essentially difference sets $A-A$, where $d^{*}(A)>0$ (as shown by Lemma \ref{lemmaMeasureCorrespondence}).  We prove Theorem \ref{theoremMainDynamic} by improving the result of \cite{GriesmerBohr} in the special case $p=2$, with a simpler proof than in \cite{GriesmerBohr}.  The following theorem summarizes the result of our construction.

\begin{theorem}\label{theoremMainDensity}
  For all $\varepsilon>0$, there is a set $A\subseteq G_{2}$ such that $d^{*}(A)>\frac{1}{2}-\varepsilon$, and $A-A$ does not contain a set of the form $B-B+g$, where $B\subseteq G_{2}$ is piecewise syndetic and $g\in G_{2}$.
\end{theorem}

Theorems \ref{theoremMainDynamic} and \ref{theoremMainDensity}  are proved in Sections \ref{sectionGroups} and \ref{sectionProof}, which are mostly self-contained.  The proof of Theorem \ref{theoremMainDensity} is a combinatorial construction requiring no background regarding dynamical systems.

%A (possibly) weaker version of Theorem \ref{theoremMainDensity} is proved in Theorem 1.2 of \cite{GriesmerBohr}, where the conclusion is merely that $A-A$ does not contain a set of the form $C+g$, where $C$ is a Bohr neighborhood of $0\in G_{2}$.

\begin{remark} When $C$ is a Bohr neighborhood of $0$ in a countable abelian group $G$, it contains a difference set $B-B$ where $B$ is syndetic (see Remark \ref{remarkSyndetic}), so the present Theorem \ref{theoremMainDensity} implies the special case of \cite[Theorem 1.2]{GriesmerBohr} where $p=2$.  However, we do not know whether piecewise syndeticity of $B\subseteq G_{2}$ implies $B-B$ contains a Bohr neighborhood -- this is equivalent to Part (i) of Question \ref{questionMain} in the case $G=G_{2}$. So we cannot be  certain that Theorem \ref{theoremMainDensity} is a strict improvement over \cite[Theorem 1.2]{GriesmerBohr}.
\end{remark}

In light of a correspondence principle (see Lemma \ref{lemmaMeasureCorrespondence}), we may also state Theorem \ref{theoremMainDensity} as a more quantitative form of Theorem \ref{theoremMainDynamic}.

\begin{theorem}\label{theoremPreciseDynamic}
  For all $\varepsilon>0$, there is
\begin{enumerate}
     \item[$\bullet$] a measure preserving action $T$ of $G_{2}$ on a probability measure space $(X,\mu)$,
     \item[$\bullet$] a set $D\subseteq X$ such that $\mu(D)>\frac{1}{2}-\varepsilon$, and
      \item[$\bullet$] a set $S\subseteq G_{2}$ such that every translate of $S$ is a set of topological recurrence,
\end{enumerate} such that $D\cap T^{g}D=\varnothing$ for all $g\in  S$.
\end{theorem}

\subsection{Sumsets}

If $G$ is a countable abelian group and $A, B \subseteq G$ have positive upper Banach density, then $A+B$ is piecewise syndetic.  For $G=\mathbb Z$, this result is due to Renling Jin \cite{JinSP}.  Bergelson, Furstenberg, and Weiss \cite{BFW} strengthened the conclusion from ``piecewise syndetic" to ``piecewise Bohr" (see Section \ref{sectionPWBohr}). In \cite{GriesmerIsr} and \cite{GriesmerThesis}, the author generalized these results to cases where $d^{*}(A)=0$ and $d^{*}(B)>0$, and these results have been generalized to other settings -- see  \cite{BeiglbockUltrafilter,BBF,BjorklundFish,Dinasso,DinassoEtAl}.  The proofs and examples in \cite{GriesmerIsr} raised the following question, stated for $G=\mathbb Z$ as \cite[Question 5.1]{GriesmerIsr}.

\begin{question}\label{questionIsr}  Let $G$ be a countable abelian group and $S\subseteq G$. Let $\tilde{S}$  be the closure of $S$ in $bG$, the Bohr compactification\footnote{We will not use the Bohr compactification in this article except in reference to Question \ref{questionIsr} -- see Section \ref{sectionBohrCompactification} for a brief discussion.} of $G$.  Let $m_{bG}$ denote Haar measure in $bG$.  Which, if any, of the following implications are true?
\begin{enumerate}
  \item[(1)] If $m_{bG}(\tilde{S}) > 0$ and $d^{*}(A)>0$, then $S+A$ is piecewise syndetic.

  \item[(2)] If $m_{bG}(\tilde{S}) > 0$ and $d^{*}(A)>0$, then $S+A$ is piecewise Bohr.

  \item[(3)] If $S$ is dense in the Bohr topology of $G$ and $d^{*}(A)>0$, then $d^{*}(S+A)=1$.
\end{enumerate}
\end{question}

Theorem 1.4 of \cite{GriesmerBohr} provides counterexamples to implications (2) and (3) when $G=G_{p}:=\bigoplus_{n=1}^{\infty} \mathbb Z/p\mathbb Z$ for some prime $p$.  The construction in the proof of the present Theorem \ref{theoremMainDensity} provides counterexamples to all three implications for $G=G_{2}$, resulting in the following theorem.  All parts of Question \ref{questionIsr} remain open for all countably infinite abelian groups besides $G_{p}$ for some prime $p$.

\begin{theorem}\label{theoremPWsynd}
  For all $\varepsilon>0$, there are sets $S, A\subseteq G_{2}:=\bigoplus_{n=1}^{\infty} \mathbb Z/2\mathbb Z$ such that $d^{*}(A)> 1-\varepsilon$, every translate of $S$ is a set of topological recurrence, and $S+A$ is not piecewise syndetic.
\end{theorem}

We now explain why Theorems \ref{theoremMainDynamic} and \ref{theoremPreciseDynamic} let us resolve Question \ref{questionIsr} for $G=G_{2}$.  This explanation uses the notion of ``Bohr recurrence" from Definition \ref{defRecurrence}.

Lemma \ref{lemmaExpanding}, together with the straightforward implication ``($S$ is a set of topological recurrence) $\implies$ ($S$ is a set of Bohr recurrence)"  shows that if every translate of $S$ is a set of topological recurrence, then $S$ is dense in the Bohr topology of $G_{2}$. The condition that $S$ is dense in the Bohr topology of $G_{2}$ implies $\tilde{S}=bG_{2}$ in the notation of Question \ref{questionIsr}.  In particular, $m_{bG_{2}}(\tilde{S}) = 1$, so Theorem \ref{theoremPWsynd} provides counterexamples to all parts of Question \ref{questionIsr} for $G= G_{2}$.

The next section surveys some notions of recurrence for dynamical systems and formulates some questions related to Theorems \ref{theoremMainDynamic} and \ref{theoremMainDensity}.  Section \ref{sectionGroups} introduces the definitions and notation needed for our constructions, and Section \ref{sectionProof} contains the proofs of Theorems \ref{theoremMainDynamic}, \ref{theoremMainDensity}, \ref{theoremPreciseDynamic}, and \ref{theoremPWsynd}.  Section \ref{sectionAppendix} contains some standard lemmas needed to relate statements about difference sets to statements about dynamical systems.   Section \ref{sectionAppendix} contains some standard lemmas needed to keep the article self-contained.  These are mostly for exposition, with the exception of Lemma \ref{lemmaMeasureCorrespondence}, which is used only to derive Theorems \ref{theoremMainDynamic} and \ref{theoremPreciseDynamic} from the proof of Theorem \ref{theoremMainDensity}.  Sections \ref{sectionGroups} and \ref{sectionProof} can be read essentially independently of the others.

\section{Two hierarchies of single recurrence properties}\label{sectionQuestions}

\subsection{Measurable, topological, and Bohr recurrence}\label{sectionDefinitions} We begin with some standard definitions.

Let $G$ be a countable abelian group.

A \emph{measure preserving $G$-system} (or briefly, ``$G$-system") is a triple $(X,\mu,T)$, where $(X,\mu)$ is a probability measure space and $T$ is an action of $G$ on $X$ preserving $\mu$, meaning $\mu(T^{g}D) = \mu(D)$ for every measurable $D\subseteq X$ and $g\in G$.  We say that $(X,\mu,T)$ is \emph{ergodic} if for all measurable $D \subseteq X$ satisfying $\mu(D\triangle T^{g}D)=0$ for all $g\in G$, $\mu(D)=0$ or $\mu(D)=1$. %For a set $D\subseteq X$ having $\mu(D)>0$ and a number $c\geq 0$, let $R(D):=\{g\in G: \mu(D\cap T^{g}D)>0\}$, and let $R_{c}(D):=\{g\in G: \mu(D\cap T^{g}D)>c\}$.

A \emph{topological $G$-system} is a pair $(X,T)$, where $X$ is a compact metric space and $T$ is an action of $G$ on $X$ by homeomorphisms.  We say that $(X,T)$ is \emph{minimal} if for all $x\in X$, the orbit $\{T^{g}x:g\in G\}$ is dense in $X$.

A \emph{group rotation $G$-system} is a pair $(Z,R_{\rho})$ where $Z$ is a compact abelian group, $\rho: G\to Z$ is a homomorphism, and $R_{\rho}^{g}z = z + \rho(g)$ for $z\in Z$ and $g\in G$.  Such a system $(Z,R_{\rho})$ may be considered as a topological $G$-system, or as a measure preserving $G$-system $(Z,m,R_{\rho})$, where $m$ is normalized Haar measure on $Z$.  The topological $G$-system $(Z,R_{\rho})$ is minimal if and only if the measure preserving $G$-system $(Z,m,R_{\rho})$ is ergodic if and only if $\rho(G)$ is dense in $Z$.

\begin{definition}\label{defRecurrence}
We say that a set $S\subseteq G$ is a

\begin{enumerate}
  \item[$\bullet$]   \emph{set of measurable recurrence} if for all measure preserving $G$-systems $(X,\mu,T)$ and every $D\subseteq X$ having $\mu(D)>0$, there exists $g\in S$ such that $\mu(D\cap T^{g}D)>0$.

  \item[$\bullet$] \emph{set of strong recurrence} if for all measure preserving $G$-systems $(X,\mu,T)$ and every $D\subseteq X$ having $\mu(D)>0$, there exists $c>0$ such that $\{g\in S: \mu(D\cap T^{g}D)>c\}$ is infinite.

  \item[$\bullet$]   \emph{set of optimal recurrence} if for all measure preserving $G$-systems $(X,\mu,T)$, every measurable $D\subseteq X$ and $c<\mu(D)^{2}$, there is a $g\in S$ such that $\mu(D\cap T^{g}D)> c$.

  \item[$\bullet$] \emph{set of topological recurrence} if for every minimal topological $G$-system $(X,T)$ and every open nonempty $U\subseteq X$, there exists $g\in S$ such that $U\cap T^{g}U \neq \varnothing$.

  \item[$\bullet$] \emph{set of Bohr recurrence} if for every minimal group rotation $G$-system $(Z,R_{\rho})$ and every open nonempty $U\subseteq Z$, there exists $g\in S$ such that $U\cap R_{\rho}^{g}U \neq \varnothing$.
\end{enumerate}
Let $\mathcal S^{1}$ be the group $\{z\in \mathbb C: |z|=1\}$, the complex numbers of modulus $1$ with the group operation of multiplication. A set $S\subseteq G$ is \emph{equidistributed} if there is a sequence of finite subsets $S_{j} \subseteq S$ such that for every non-constant homomorphism $\chi: G\to \mathcal S^{1}$, $\lim_{j\to \infty} \frac{1}{|S_{j}|} \sum_{g\in S_{j}} \chi(g) = 0$. \end{definition}

We abbreviate the above definitions in the following conditions.

\begin{enumerate}
  \item[($R_{1}$)]  $S$ is equidistributed.

  \item[($R_{2}$)] $S$ is a set of optimal recurrence.

  \item[($R_{3}$)]  $S$ is a set of strong recurrence.

  \item[($R_{4}$)] $S$ is a set of measurable recurrence.

  \item[($R_{5}$)] $S$ is a set of topological recurrence.

  \item[($R_{6}$)] $S$ is a set of Bohr recurrence.
\end{enumerate}

We have  $(R_{i}) \implies (R_{i+1})$ for $i=1,\dots, 5$.  These implications are well known -- see Section \ref{sectionImplications} for a proof of $(R_{1}) \implies (R_{2})$ and further discussion.  We briefly summarize what is known regarding the reverse implications $(R_{j}) \implies (R_{i})$ for $j>i$. Here ``group" means ``countably infinite abelian group" and $G_{p}$ denotes $\bigoplus_{n=1}^{\infty} \mathbb Z/p\mathbb Z$.

The question of whether $(R_{6}) \implies (R_{5})$ is well known and remains open for every group $G$. The question was first explicitly asked in this form by Katznelson \cite{KatznelsonChromatic} for $G=\mathbb Z$, but the question is older -- see \cite{Weiss}, as well as \cite{GlasnerPolish} and \cite{BoshernitzanGlasner} for exposition and a related problem.

For $G=\mathbb Z$, K\v{r}\'{i}\v{z} \cite{Kriz} proved that $(R_{5})\cnot\implies (R_{4})$, and Forrest \cite{ForrestThesis} adapted K\v{r}\'{i}\v{z}'s example to $G= G_{2}$.  Randall McCutcheon \cite{McCutcheonAlexandria, McCutcheonBook} presented a simplification of K\v{r}\'{i}\v{z}'s example due to Imre Ruzsa.   Our proof of Theorem \ref{theoremMainDynamic} provides another proof that $(R_{5})\cnot\implies(R_{4})$ for $G=G_{2}$. Whether $(R_{5})\implies (R_{4})$ is open for all groups $G$ besides $G_{2}$ and $\mathbb Z$.

For $G=\mathbb Z$ and $G= G_{2}$, Forrest \cite{ForrestPaper} proved that $(R_{4})\cnot\implies (R_{3})$ and $(R_{3})\cnot\implies (R_{2})$.  For $G=\mathbb Z$, McCutcheon \cite{McCutcheonAlexandria} provides a simplification of Forrest's construction and \cite{GriesmerRRPD} provides another proof of $(R_{4})\cnot\implies (R_{3})$.  The status of the implications $(R_{4}) \implies (R_{3})$ and $(R_{3}) \implies (R_{2})$ is unkown for all other groups $G$.

For $G= \mathbb Z$, the classical example $S=\{2n: n\in \mathbb Z\}$ shows that $(R_{2}) \cnot\implies (R_{1})$.  Constructing examples showing that $(R_{2})\cnot\implies (R_{1})$ for an arbitrary countably infinite abelian group is not difficult, but it makes an interesting exercise.

\subsection{Translations}\label{sectionTranslations} If $G$ is a countable abelian group and $S\subseteq G$, we say that $S$ satisfies property $(R_{i}^{\bullet})$ if every translate of $S$ satisfies property $(R_{i})$, meaning $S+g$ satisfies $(R_{i})$ for all $g\in G$.  It is easy to verify that $(R_{1}^{\bullet}) \iff (R_{1})$, while $(R_{i}) \cnot\implies (R_{i}^{\bullet})$ for each $i>1$ and every nontrivial group $G$.  Observe that $(R_{i}^{\bullet}) \implies (R_{j})$ if and only if $(R_{i}^{\bullet}) \implies (R_{j}^{\bullet})$.

We say that $S$ satisfies property $(PR_{i})$ if $S\setminus \{0\}$ satisfies property $(R_{i})$.

We summarize what is currently known regarding the implications $(R_{i}^{\bullet}) \implies (R_{j})$.

For every countable abelian group $G$, $(R_{6}^{\bullet})\cnot\implies (R_{1})$.  For $G=\mathbb Z$ this is due to Katznelson \cite[Theorem 2.2]{KatznelsonStockholm}, and for general $G$ to Saeki \cite{Saeki}, by way of constructing sets $S\subseteq G$ dense in the Bohr topology that do not satisfy $(R_{1})$.  See Lemma \ref{lemmaExpanding} for a proof that such constructions prove $(R_{6}^{\bullet})\cnot\implies (R_{1})$.

For $G=\mathbb Z$, whether the implication $(R_{6}^{\bullet}) \implies (R_{4})$ holds has been asked\footnote{In \cite{BergelsonRuzsa}, \cite{RuzsaBook}, and \cite{HegyvariRuzsa}, the question is phrased as ``If $A\subseteq \mathbb Z$ has positive upper Banach density, must $A-A$ contain a Bohr neighborhood of some $n\in \mathbb Z$?"  We discuss this form of the question in Section \ref{sectionCombinatorial}.} in \cite{BergelsonRuzsa}, \cite{RuzsaBook}, \cite{GriesmerIsr}, and \cite{HegyvariRuzsa}.  The problem remains stubbornly open. For $G_{p}=\bigoplus_{n=1}^{\infty} \mathbb Z/p \mathbb Z$, where $p$ is prime, the author  proved in \cite{GriesmerBohr} that $(R_{6}^{\bullet})\cnot\implies (R_{4})$ by exhibiting a set $A\subseteq G_{p}$ having $d^{*}(A)>0$ such that $A-A$ does not contain a Bohr neighborhood of any $g\in G_{p}$.  See Lemma \ref{lemmaDifferenceAndRecurrence} for an explanation of why said construction implies $(R_{6}^{\bullet})\cnot\implies (R_{4})$.

For $G_{2} = \bigoplus_{n=1}^{\infty} \mathbb Z/2\mathbb Z$, Theorem \ref{theoremMainDynamic} of the present article says that $(R_{5}^{\bullet}) \cnot\implies (R_{4})$.

For $G = \mathbb Z$, the main result of \cite{GriesmerRRPD} shows that $(R_{4}^{\bullet}) \cnot\implies (R_{3})$, and in fact there is a set $S\subseteq \mathbb Z$ satisfying $(R_{4}^{\bullet})$ such that no translate of $S$ satisfies $(R_{3})$.

The implications listed above are all that are currently known, so the following questions remain.  A priori, the answer to each part could be different for different groups $G$.

\begin{question}\label{questionMain}  \begin{enumerate}
   \item[(i)]  Does $(R_{6}^{\bullet}) \implies (R_{5})$?   (Open for every $G$.)

  \item[(ii)]  Does $(R_{6}^{\bullet}) \implies (R_{4})$?  (Open for every $G$ except $G_{p}$ where $p$ is prime.)

  \item[(iii)]   Does $(R_{5}^{\bullet}) \implies (R_{4})$?  (Open for every $G$ except $G_{2}$.)

  \item[(iv)]  Does  $(R_{4}^{\bullet}) \implies (R_{3})$?  (Open for every $G$ except $\mathbb Z$.)

  \item[(v)]  Does $(R_{3}^{\bullet}) \implies (R_{2})$?  (Open for every $G$.)

  \item[(vi)]  Does $(R_{2}^{\bullet}) \implies (R_{1})$?  (Open for every $G$.)
 \end{enumerate}
Furthermore, in case $(R_{i}^{\bullet}) \cnot\implies (R_{j})$, is there a set $S\subseteq G$ satisfying $(R_{i}^{\bullet})$ while no translate of $S$ satisfies $(PR_{j})$?
\end{question}

We expect that for every  $G$, the answers to parts (ii) through (vi) are all ``no." We reserve speculation on part (i), as a negative answer would resolve the difficult question of whether $(R_{6})\implies (R_{5})$, while a positive answer would be surprising.

If the answer to a given part does not depend on the group $G$, it would be interesting to identify a general principle which implies that the answer must be the same for every $G$.

Finally, while our Theorem \ref{theoremMainDynamic} shows that $(R_{5}^{\bullet}) \cnot\implies(R_{4})$ for $G=G_{2}$, our proof does not provide an example of a set $S$ satisfying $(R_{5}^{\bullet})$ such that no translate of $S$ satisfies $(PR_{4})$.  In fact, for the set $S$ constructed in the proof, Lemma \ref{lemmaPoincare}  implies $S+\mathbf{1}$ satisfies $(PR_{4})$ -- see Section \ref{sectionGroups} for notation.

As mentioned at the beginning of this subsection, $(R_{6}^{\bullet}) \cnot\implies (R_{1})$ is established for all countable abelian groups $G$ in \cite{Saeki}, so Theorem \ref{theoremMainDynamic} is a refinement of that result for the group $G_{2}$.

\subsection{Syndeticity and piecewise syndeticity}\label{sectionCombinatorial}
In this and the following subsections we formulate some definitions needed to interpret Question \ref{questionMain} in terms of difference sets.  We fix a countable abelian group $G$ for the remainder of this section.
\begin{definition}\label{definitionSyndetic}
 A set $A\subseteq G$ is

  \begin{enumerate}
    \item[$\bullet$] \emph{thick} if for every finite $F\subseteq G$, there exists $g\in G$ such that $F+g \subseteq A$,
    \item[$\bullet$]  \emph{syndetic} if there is a finite set $F\subseteq G$ such that $A+F = G$, meaning $G$ is a union of finitely many translates of $A$,
    \item[$\bullet$] \emph{piecewise syndetic} if there is a syndetic set $S\subseteq G$ such that for all finite $F\subseteq S$, there exists $g\in G$ such that $F+g\subseteq A$.
 \end{enumerate}

  \end{definition}

\begin{remark}\label{remarkPWSdefinition}
 Our definition of ``piecewise syndetic" is not standard, and Lemma \ref{lemmaPWSequivalents} shows that it is equivalent to the standard one.  We use our definition so that it is easy to see that if $A$ is piecewise syndetic, then $A-A$ contains a set of the form $S-S$, where $S$ is syndetic.
\end{remark}

\begin{remark}
 A set $A\subseteq G$ is thick if and only if $d^{*}(A)=1$.
\end{remark}
\subsection{The Bohr topology}\label{sectionBohr}  If $G$ is a countable abelian group, the Bohr topology on $G$ is the weakest topology on $G$ such that every homomorphism $\rho: G\to \mathbb T$ is continuous, where $\mathbb T=\mathbb R/\mathbb Z$ with the usual topology.  A set $S\subseteq G$ satisfies $(R_{6})$ if and only if $0$ is in the closure of $S$ in the Bohr topology, and $S$ satisfies $(R_{6}^{\bullet})$ if and only if $S$ is dense in the Bohr topology -- see Lemmas \ref{lemmaBohr} and \ref{lemmaExpanding} for proofs.  The open sets of the Bohr topology are called \emph{Bohr neighborhoods}.  The group operation and inversion are both continuous in the Bohr topology, so $U$ is a Bohr neighborhood of $g\in G$ if and only if $U-g$ is a Bohr neighborhood of $0$.  See \cite{RudinFourier} for exposition of the Bohr topology of locally compact abelian groups, including countable discrete groups.

The Bohr topology is the weakest topology making every homomorphism $\rho:G\to Z$ to a compact abelian group $Z$ continuous.

A neighborhood base for $0$ in the Bohr topology on $\mathbb Z$ is the collection of sets of the form $\{n: \max_{1\leq i \leq d} \|s_{i}n\| < \varepsilon\}$, where $s_{i}\in \mathbb R$, $\varepsilon>0$, and $\|x\|$ denotes the distance from $x$ to the nearest integer.

For a fixed prime number $p$, a neighborhood base for $0$ in the Bohr topology on $G_{p}:=\bigoplus_{n=1}^{\infty} \mathbb Z/p\mathbb Z$ is the collection of finite index subgroups of $G_{p}$.  A set $U\subseteq G_{p}$ is open in the Bohr topology if and only if it is a union of cosets of finite index subgroups of $G_{p}$.

\subsection{Piecewise Bohr sets}\label{sectionPWBohr} A set $S\subseteq G$ is \emph{piecewise Bohr} if there is a nonempty Bohr neighborhood $U\subseteq G$ such that for every finite $F\subseteq U$, there is a $g\in G$ such that $g+F\subseteq S$.  Equivalently, a set $S\subseteq G$ is piecewise Bohr if there is a nonempty Bohr neighborhood $U$ of some $g$ and a thick set $C$ such that $U\cap C \subseteq S$ (see Definition \ref{definitionSyndetic}).  While we do not need the equivalence of these definitions, a proof can be found in \cite{GriesmerIsr} for the case $G=\mathbb Z$.

\begin{remark}\label{remarkSyndetic}
  Bohr neighborhoods are syndetic, due to the compactness of $\mathbb T^{d}$ for each $d$, and every Bohr neighborhood of $0$ contains a difference set $A-A$ where $A$ is a Bohr neighborhood of $0$.  It follows that if $A$ is piecewise Bohr, then $A-A$ contains a set of the form $S-S$, where $S$ is syndetic.
\end{remark}

\subsection{Bohr compactification}\label{sectionBohrCompactification}  The Bohr compactification $bG$ of a locally compact abelian group $G$ is the unique compact abelian group $H$ such that $G$ embeds densely in $H$ (so that $G$ may be identified with a subgroup $\tilde{G}$ of $H$) and every continuous homomorphism to a compact abelian group $\rho:G\to Z$ has a unique continuous extension $\tilde{\rho}:H\to Z$.   The Bohr topology on $G$ is the topology $G$ inherits from $bG$ as a subspace.  We will not use the Bohr compactification in this article, but Question \ref{questionIsr} is reproduced from \cite{GriesmerIsr}, where it is mentioned.  See \cite{RudinFourier} for further exposition of $bG$.

\subsection{Upper Banach density}\label{sectionUBD}  A \emph{F{\o}lner sequence} for an abelian group $G$ is a sequence of finite subsets $\Phi_{n}\subseteq G$ such that $\lim_{n\to \infty} \frac{|(\Phi_{n}+g)\cap \Phi_{n}|}{|\Phi_{n}|}=1$ for every $g\in G$.  Every countable abelian group admits a F{\o}lner sequence.\footnote{The standard way to see this is to appeal to some theory of  amenable groups: a countable discrete group is amenable if and only if it admits a F{\o}lner sequence, and abelian groups are amenable by the Markov-Kakutani fixed point theorem.  However, given a specific abelian group, it is usually possible to construct a F{\o}lner sequence by hand.}

If $\mathbf \Phi = (\Phi_{n})_{n\in \mathbb N}$ is a F{\o}lner sequence for $G$ and $A\subseteq G$,  the \emph{upper density of $A$ with respect to $\mathbf{\Phi}$} is $\bar{d}_{\mathbf{\Phi}}(A):=\limsup_{n\to \infty}\frac{|A\cap \Phi_{n}|}{|\Phi_{n}|}$; we write $d_{\mathbf{\Phi}}(A)$ if the limit exists.  The \emph{upper Banach density} of $A$ is $d^{*}(A):=\sup\{\bar{d}_{\mathbf{\Phi}}(A): \mathbf{\Phi} \text{ is a F{\o}lner sequence}\}$.  Note that for every $A\subseteq G$, there is a F{\o}lner sequence $\mathbf{\Phi}$ such that $d^{*}(A)=d_{\mathbf{\Phi}}(A)$.

If $(H_{n})_{n\in \mathbb N}$ is an increasing sequence of finite subgroups of $G$ such that $G= \bigcup_{n=1}^{\infty} H_{n}$, then $(H_{n})_{n\in \mathbb N}$ is a F{\o}lner sequence for $G$.

While upper Banach density is not finitely additive, it enjoys the following weaker property.

\begin{lemma}\label{lemAdditivity}
  Let $G$ be a countable abelian group, $g\in G$, and $A\subseteq G$. If $A\cap (A+g)=\varnothing$, then $d^{*}(A\cap (A+g))=2d^{*}(A)$.
\end{lemma}
We omit the proof, which is a straightforward application of the relevant definitions.

\subsection{Difference sets and recurrence}
The study of measurable and topological recurrence is closely tied to the study of combinatorial structure in difference sets $A-A$.  After the following definitions, we state Question \ref{questionDifferenceSets} to rephrase Parts (i)-(iii) of Question \ref{questionMain} in terms of difference sets.

Fix a countably infinite abelian group $G$ for the remainder of this subsection.

\begin{definition}
     We say that $S \subseteq G$ is a
  \begin{enumerate}
    \item[$\bullet$] \emph{set of chromatic recurrence} if for all $k\in \mathbb N$ and every partition $G=\bigcup_{j=1}^{k} A_{j}$ of $G$, there is a $j\leq k$ such that $(A_{j}-A_{j}) \cap S \neq \varnothing$,
    \item[$\bullet$]  \emph{set of density recurrence} if for every $A\subseteq G$ having $d^{*}(A)>0$, $(A-A)\cap S\neq \varnothing$.
  \end{enumerate}
\end{definition}

The following proposition is a special case of Theorems 2.2 and 2.6 of \cite{BergelsonMcCutcheonRecurrence}.
\begin{proposition}\label{propositionCorrespondence} Let $S\subseteq G$.  \begin{enumerate}
    \item[(i)] $S$ is a set of measurable recurrence if and only if $S$ is a set of density recurrence.
    \item[(ii)] $S$ is a set of topological recurrence if and only if $S$ is a set of chromatic recurrence.
  \end{enumerate}
\end{proposition}

We also need the following lemma, which follows immediately from our definition of ``piecewise syndetic" and the partition regularity of piecewise syndeticity (Lemma \ref{lemmaPWPartitionRegular}).

\begin{lemma}\label{lemmaPWrecurrence}
A set $S\subseteq G$ is a set of chromatic recurrence if and only if  $(A-A)\cap S \neq \varnothing$ whenever $A$ is piecewise syndetic.
\end{lemma}

The first three parts of Question \ref{questionDifferenceSets} rephrase Parts (i)-(iii) of Question \ref{questionMain} in terms of difference sets.  Part (iv) rephrases the question of whether $(R_{6})\implies (R_{5})$ in terms of difference sets.

\begin{question}\label{questionDifferenceSets}
  Let $G$ be a countable abelian group and $A \subseteq G$.

  \begin{enumerate}
    \item[(i)]  Does $A$ being piecewise syndetic imply that $A-A$ contains a Bohr neighborhood of some $g\in G$?  (Open for every $G$.)
    \item[(ii)] Does $d^{*}(A)>0$ imply that $A-A$ contains a Bohr neighborhood of some $g\in G$?  (Open for every $G$ except $G_{p}$ for prime $p$, the main result of \cite{GriesmerBohr} shows that the answer is ``no" for these $G_{p}$.)
    \item[(iii)]  Does $d^{*}(A)>0$ imply that $A-A$ contains a set of the form $B-B+g$, where $g\in G$ and $B\subseteq G$ is piecewise syndetic?  (Open for every $G$ except $G_{2}$, where Theorem \ref{theoremMainDensity} provides a negative answer.) 
    \item[(iv)]  Does $A$ being piecewise syndetic imply $A-A$ contains a Bohr neighborhood of $0\in G$?  (Open for every $G$.)\end{enumerate}
\end{question}

Interpreting the assertion $(R_{4}^{\bullet}) \cnot\implies (R_{3})$ in terms of difference sets requires more intricate statements than those in Question \ref{questionDifferenceSets}, see \cite{GriesmerRRPD} for details in the case $G=\mathbb Z$.

The next lemma proves that Parts (i)-(iii) of Question \ref{questionDifferenceSets} are really equivalent to the corresponding parts of Question \ref{questionMain}, and that Part (iv) of Question \ref{questionDifferenceSets} is equivalent to the question ``does $(R_{6})\implies (R_{5})$?".  Lemma \ref{lemmaRecurrenceAndRecurrence} provides a similar reformulation in terms of sets of return times.

\begin{lemma}\label{lemmaDifferenceAndRecurrence}  Let $G$ be a countable abelian group and $A\subseteq G$.
  \begin{enumerate}
    \item[(i)]  ($A$ is piecewise syndetic $\implies$ $A-A$ contains  Bohr neighborhood of some $g\in G$) if and only if $(R_{6}^{\bullet}) \implies (R_{5})$.

    \item[(ii)] ($d^{*}(A)>0$ $\implies$ $A-A$ contains a Bohr neighborhood of some $g\in G$) if and only if $(R_{6}^{\bullet}) \implies (R_{4})$.

        \item[(iii)] ($d^{*}(A)>0$ $\implies$ $A-A$ contains a set $g+B-B$, where $B\subseteq G$ is piecewise syndetic and $g\in G$) if and only if $(R_{5}^{\bullet}) \implies (R_{4})$.

        \item[(iv)]      ($A$ is piecewise syndetic $\implies $$A-A$ contains a Bohr neighborhood of $0\in G$) if and only if $(R_{6})\implies (R_{5})$.
\end{enumerate}
\end{lemma}

\begin{proof}
We prove only (iv).  The other equivalences are proved similarly.

  First suppose that if $A-A$ contains a Bohr neighborhood of $0\in G$ whenever $A$ is piecewise syndetic, and that $S\subseteq G$ satisfies $(R_{6})$.  Then $S\cap (A-A)\neq \varnothing$ whenever $A$ is piecewise syndetic, by  Lemma \ref{lemmaBohr}, and $S$ therefore satisfies $(R_{5})$, by Proposition \ref{propositionCorrespondence} and Lemma \ref{lemmaPWrecurrence}.

Now suppose $(R_{6})\implies (R_{5})$, and that $A\subseteq G$ is piecewise syndetic.  Assume, to get a contradiction, that $A-A$ does not contain a Bohr neighborhood of $0$.  Then $(A-A)^{c}:=G\setminus (A-A)$ has nonempty intersection with every Bohr neighborhood of $0$, so that $(A-A)^{c}$ satisfies $(R_{6})$, by Lemma \ref{lemmaBohr}.  Since we are assuming $(R_{6})\implies (R_{5})$, we conclude that $(A-A)^{c}$ satisfies $(R_{5})$, so that Proposition \ref{propositionCorrespondence} and Lemma \ref{lemmaPWrecurrence} imply $(A-A)^{c}\cap (A-A)\neq \varnothing$, a contradiction.
\end{proof}

Let $X$ be a set and $T$ an action of $G$ on $X$.  If $D\subseteq X$ let $\operatorname{Ret}_{T}(D):=\{g\in G: D\cap T^{g}D\neq \varnothing\}$.  The following lemma provides equivalent formulations of Parts (i)-(iii) of Question \ref{questionMain} in terms of the sets $\operatorname{Ret}_{T}(D)$.

\begin{lemma}\label{lemmaRecurrenceAndRecurrence}
Let $G$ be a countable abelian group.
  \begin{enumerate}

    \item[(i)]  (For every minimal $G$-system $(X,T)$ and every open $U\subseteq X$, $\operatorname{Ret}_{T}(U)$ contains a Bohr neighborhood of some $g\in G$) if and only if $(R_{6}^{\bullet}) \implies (R_{5})$

        \item[(ii)] (For every measure preserving $G$-system $(X,\mu,T)$ and  $D\subseteq X$ having $\mu(D)>0$, the set $\operatorname{Ret}_{T}(D)$ contains a Bohr neighborhood of some $g\in G$) if and only if $(R_{6}^{\bullet}) \implies (R_{4})$.

        \item[(iii)] (For every measure preserving $G$-system $(X,\mu,T)$ and $D\subseteq X$ having $\mu(D)>0$, the set $\operatorname{Ret}_{T}(D)$ contains a set of the form $B-B+g$, where $B$ is piecewise syndetic) if and only if $(R_{5}^{\bullet}) \implies (R_{4})$.

        \item[(iv)]  (For every minimal $G$-system $(X,T)$ and every open $U\subseteq X$, $\operatorname{Ret}_{T}(U)$ contains a Bohr neighborhood of $0$) if and only if $(R_{6})$ $\implies$ $(R_{5})$.

\end{enumerate}
\end{lemma}

We omit the proof of Lemma \ref{lemmaRecurrenceAndRecurrence}; like Lemma \ref{lemmaDifferenceAndRecurrence} it may be proved with straightforward applications of Lemmas \ref{lemmaBohr} and \ref{lemmaExpanding}.

%Question \ref{questionMain} (i) is equivalent to: Is there a piecewise syndetic $A\subseteq G$ such that $A-A$ does not contain a Bohr neighborhood of some $g\in G$?
%
%
%Question \ref{questionMain} (ii) is equivalent to the following:  Is there a set $A\subseteq G$ such that $d^{*}(A)>0$ and $A-A$ does not contain a Bohr neighborhood of some $g\in G$?
%
%Question \ref{questionMain} (iii) is equivalent to: is there a set $A\subseteq G$ such that $d^{*}(A)>0$ and $A-A$ does not contain a set of the form $g+B-B$, where $B\subseteq G$ is piecewise syndetic?

\subsection{Sumsets and measure expansion}

\begin{definition}
  Let $G$ be a countable abelian group and $S\subseteq G$.  We say that  $S$ is

  \begin{enumerate}
    \item[$\bullet$]  \emph{measure expanding}\footnote{Also called ``ergodic" in \cite{BFPlunnecke}.} if for every ergodic measure preserving $G$-system $(X,\mu,T)$ and every $D\subseteq X$ having $\mu(D)>0$, we have $\mu\bigl(\bigcup_{g\in S} T^{g}D\bigr)=1$.

    \item[$\bullet$] \emph{measure transitive} if for every (not necessarily ergodic) measure preserving $G$-system $(X,\mu,T)$ and every pair of sets $C, D \subseteq X$ such that $\mu(C\cap T^{g} D)>0$ for some $g\in G$, there exists $h\in S$ such that $\mu(C\cap T^{h}D)>0$.

    \item[$\bullet$]  \emph{density expanding} if for every set $A\subseteq G$ having $d^{*}(A)>0$, $d^{*}(S+A)=1$.

    \item[$\bullet$] \emph{expanding for minimal $G$-systems} if for every minimal topological $G$-system $(X,T)$ there is a dense $G_{\delta}$ set  $Y\subseteq X$, such that $\{T^{g}y: g\in S\}$ is dense in $X$ for all $y\in Y$.

    \item[$\bullet$] \emph{transitive for minimal $G$-systems} if for every minimal toplogical $G$-system and every pair of nonempty open sets $U, V\subseteq X$  there exists $h\in S$ such that $U\cap T^{h}V\neq \varnothing$.

    \item[$\bullet$] \emph{chromatically expanding} if for every piecewise syndetic set $A\subseteq G$, $d^{*}(S+A)=1$.
    \end{enumerate}

\end{definition}

We have the following equivalences and implications. The equivalences are proved in Lemma \ref{lemmaExpanding}.
\begin{align*}
  (R_{4}^{\bullet}) &\iff S \text{ is measure expanding }\\
  &\iff S \text{ is measure transitive}\\
   &\implies S \text{ is density expanding},\\[.08in]
  (R_{5}^{\bullet}) &\iff S \text{ is expanding for minimal $G$-systems }\\
  &\iff S \text{ is transitive for minimal $G$-systems}\\
   &\implies S \text{ is chromatically expanding},
\end{align*}
We now discuss the implications ``$(R_{4}^{\bullet}) \implies$ $S$ is density expanding" and ``$(R_{5}^{\bullet}) \implies$ $S$ is chromatically expanding."

It is well known that if $S$ is measure expanding, then $S$ is density expanding -- see Correspondence Principle II in \cite{BFPlunnecke} or  Proposition 2.3 of \cite{BBF}.  A counterexample to implication (3) of Question \ref{questionIsr} for a given $G$ would therefore show that $(R_{6}^{\bullet}) \cnot\implies (R_{4})$.  The implication ``$(R_{5}^{\bullet})$ $\implies$ $S$ is chromatically expanding" has a proof analogous to the proof of ``$(R_{4}^{\bullet}) \implies$ $S$ is density expanding."  Whether density expanding implies measure expanding for $G=\mathbb Z$ is asked in \cite{GriesmerIsr} and remains open.  In general we have the following question which is open for every countable abelian group $G$.

\begin{question}\label{questionExpanding}
Let $S$ be a subset of a countable abelian group $G$.
\begin{enumerate}
    \item[(i)] Does $S$ being density expanding imply $(R_{4}^{\bullet})$?

\item[(ii)]  Does $S$ being chromatically expanding imply $(R_{5}^{\bullet})$?
\end{enumerate}
\end{question}

If every density expanding $S$ satisfies $(R_{4}^{\bullet})$, then Part (1) of Question \ref{questionIsr} is equivalent to Part (ii) of Question \ref{questionMain}.  We also have the following variant of Question \ref{questionIsr}, both parts of which are open for every $G$ besides $G_{2}$.

\begin{question}\label{questionSumsets}  Let $S\subseteq G$.
\begin{enumerate}
%\item[(i)]  If $S\subseteq G$ is dense in the Bohr topology, does it follow that $d^{*}(S+A)=1$ whenever $d^{*}(A)>0$?  (Open for every $G$ not having $G_{p}$ as a quotient for some prime $p$.)

\item[(i)]  Does $(R_{5}^{\bullet})$ imply that $S$ is density expanding?

%\item[(ii)]  If $S\subseteq G$ is dense in the Bohr topology, does it follow that $S+A$ is piecewise syndetic whenever $d^{*}(A)>0$?  (Open for every $G$ not having $G_{p}$ as a quotient for some prime $p$.)

\item[(ii)]  Does $(R_{5}^{\bullet})$ imply that $S+A$ is piecewise syndetic whenever $d^{*}(A)>0$?
\end{enumerate}
\end{question}
For $G=G_{2}$, Theorem \ref{theoremPWsynd} answers both parts of Question \ref{questionSumsets} in the negative.  A positive solution to Problem \ref{problemHamming} together with the results of \cite{GriesmerBohr} would answer both parts of Question \ref{questionSumsets} in the negative for $G=G_{p}$, where $p$ is any prime.

\section{Vector spaces over \texorpdfstring{$\mathbb F_{p}$}{Fp}}\label{sectionGroups}
In this section we state some definitions and conventions needed for the proof of Theorems \ref{theoremMainDynamic}, \ref{theoremMainDensity}, \ref{theoremPreciseDynamic}, and \ref{theoremPWsynd}.  We identify a useful presentation of $G_{p}:=\bigoplus_{n=1}^{\infty} \mathbb Z/p\mathbb Z$ and some associated subgroups.  We use an arbitrary prime $p$ so that we can formulate Problem \ref{problemHamming}, but we specialize to the case $p=2$ in our proofs. This material is also presented in Section 2 of \cite{GriesmerBohr}.

 If $p$ is a prime number, let $\mathbb F_{p}$ $(= \mathbb Z/p\mathbb Z)$ denote the finite field with $p$ elements.  We write the elements of $\mathbb F_{p}$ as $0,1,\dots, p-1$.  Consider the countable direct sum $G_{p} := \bigoplus_{n=1}^{\infty} \mathbb F_{p}$.
\subsection{Presentation of \texorpdfstring{$G_{p}$}{G}.}\label{secPresentation}  Let $\Omega := \{0,1\}^{\mathbb N}$, and write elements of $\Omega$ as $\omega = \omega_{1}\omega_{2}\omega_{3}\dots$.  For each $n\in \mathbb N$ let $\Omega_{n}=\{0,1\}^{\{1,\dots,n\}}$ and $\pi_{n}:\Omega\to \Omega_{n}$ be the projection map given by $\pi_{n}(\omega) = \omega_{1}\dots\omega_{n}$.  Let $\Gamma_{p}$ be the group of functions $g:\Omega\to \mathbb F_{p}$ with the group operation of pointwise addition.

For each $n\in \mathbb N$, let $G_{p}^{(n)}$ be the subgroup of $\Gamma_{p}$ consisting of functions of the form $f\circ \pi_{n}$, where $f:\Omega_{n}\to \mathbb F_{p}$. Observing that $G_{p}^{(n)}\subseteq G_{p}^{(n+1)}$ for each $n$, we let $\tilde{G}_{p}:=\bigcup_{n\in \mathbb N} G_{p}^{(n)}$.  Then $\tilde{G}_{p}$ is a countable abelian group isomorphic\footnote{One can construct the isomorphism by hand, but when $p$ is prime it suffices to observe that both $G_{p}$ and $\tilde{G}_{p}$ are countably infinite vector spaces over the finite field $\mathbb F_p$.} to $G_{p}$.  Our constructions are easier to work with from the perspective of $\tilde{G}_{p}$ rather than the standard presentation of a countable direct sum, so from now on we let $G_{p}$ denote the group $\tilde{G}_{p}$.

%Let $\Gamma_{p}$ be the group of continuous $\mathbb F_p$-valued functions on $\Omega$, with the group operation of pointwise addition, so that $\Gamma_{p}$ is the group of functions $f:\Omega \to \mathbb F_{p}$ depending on only finitely many coordinates of $\omega\in \Omega$.  Note that $\Gamma_{p}$ is isomorphic\footnote{One can construct the isomorphism by hand, but when $p$ is prime it suffices to observe that both $G_{p}$ and $\Gamma_{p}$ are countably infinite vector spaces over the finite field $\mathbb F_p$.} to $G_{p}$.

%For each $n\in \mathbb N$, let $\Omega_{n}=\{0,1\}^{n}$.  Let $\pi_{n}:\Omega\to \Omega_{n}$ be the natural quotient map, defined by $\pi_{n}(\omega)=\omega_{1}\dots\omega_{n}$.  Let $G_{p}^{(n)}$ be the subgroup of $G_{p}$ consisting of functions of the form $f\circ \pi_{n}$, where $f:\Omega_{n}\to \mathbb F_{p}$.  Then $G_{p}^{(n)}\subseteq G_{p}^{(n+1)}$ for each $n$, and $G_{p}=\bigcup_{n\in \mathbb N}G_{p}^{(n)}$.

We observe that $G_{p}^{(n)}$ is isomorphic to $(\mathbb F_{p})^{\Omega_{n}}$, and we will identify elements of $G_{p}^{(n)}$ with elements of $(\mathbb F_{p})^{\Omega_{n}}$.  The identification is given by  $g \leftrightarrow \tilde{g}$ if and only if $g=\tilde{g}\circ \pi_{n}$ for $\tilde{g}\in (\mathbb F_{p})^{\Omega_{n}}$.

%For $g\in G_{p}^{(n)}$, $\omega\in \Omega_{n}$, and $E\subseteq \mathbb F_{p}$, we abuse notation by writing $g(\omega)$ for $\tilde{g}(\omega)$, $\omega \in \Omega$ and $g^{-1}(E)$ for $\tilde{g}^{-1}(E)$.  While $g\in G_{p}^{(n)}$ may be considered as an element of $G_{p}^{(m)}$ for $m>n$, we confine these conventions to cases where there is no ambiguity.

Let $G_{p}^{(0)}$ denote the group of constant functions $f:\Omega\to \mathbb F_{p}$, so that $G_{p}^{(0)}\subseteq G_{p}^{(1)}$. Let $\mathbf 1 \in G_{p}$ denote the constant function where $\mathbf 1(\omega) = 1 \in \mathbb F_{p}$ for every $\omega \in \Omega$. For $x\in \mathbb F_{p}$, define $x\mathbf 1$ to be the constant function having $x\mathbf 1(\omega)= x$ for all $\omega \in \Omega$, and let $\mathbf{0}$ denote $0\mathbf{1}$, the identity element of $G_{p}$.

\begin{remark}
 $G_{p}$ is the group of continuous functions $g:\Omega\to\mathbb F_{p}$, where $\Omega$ has the product topology and $\mathbb F_{p}$ has the discrete topology.  Alternatively, $G_{p}$ is the group of functions $g:\Omega\to \mathbb F_{p}$ where $g(\omega)$ depends on only finitely many coordinates of $\omega$.
\end{remark}

\subsection{Cylinder sets}\label{secConvention} If $\tau\in \Omega_{n}$, let $[\tau]\subseteq \Omega$ be $\pi_{n}^{-1}(\tau)$, so that $[\tau]:=\{\omega\in \Omega: \omega_{i}=\tau_{i} \text{ for } 1\leq i \leq n\}$.  We call $[\tau]$ a \emph{cylinder set}.  Each cylinder set $[\tau]$ is homeomorphic to $\Omega$ by the map $\theta:[\tau]\to \Omega$, $\theta(\omega):=\omega_{n+1}\omega_{n+2}\dots$.

Observe that $G_{p}^{(n)}$ is the group of functions $g:\Omega\to \mathbb F_{p}$ which are constant on the cylinder sets $[\tau]$, where $\tau\in \Omega_{n}$.

\begin{definition}\label{defContent}
  If $E\subseteq \Omega$, let $|E|_{n}:=|\{\tau \in \Omega_{n}: [\tau]\subseteq E\}|$.
\end{definition}
The above definition will usually be applied to sets of the form $g^{-1}(S)$ where $g\in G_{p}^{(n)}$ and $S\subseteq\mathbb F_{p}$.  We list some relevant properties.

\begin{observation}\label{obsContent}

\begin{enumerate}
\item[(C1)]  $|E|_{n}\leq 2^{n}$ for all $E\subseteq\Omega$.

\item[(C2)]  For an element $g\in G_{p}^{(n)}$, if $g=\tilde{g}\circ \pi_{n}$, then $|g^{-1}(1)|_{n} = |\tilde{g}^{-1}(1)|$.

\item[(C3)] If $g\in G_{p}^{(n)}$, $A, B\subseteq \mathbb F_{p}$, and $A\cap B=\varnothing$, then
\[|g^{-1}(A)\cup g^{-1}(B)|_{n}= |g^{-1}(A)|_{n}+|g^{-1}(B)|_{n}.\]
\end{enumerate}

\end{observation}

%We may also view $G_{p}^{(n)}$ as the subgroup of $G_{p}$ consisting of functions which depend only on the first $n$ coordinates of $\omega\in \Omega$.  If $g\in G_{p}^{(n)}$, the set $g^{-1}(1):=\{\omega\in \Omega : g(\omega) = 1\}$ may be identified with the set $\{\tau \in \Omega_{n}: g([\tau])=\{1\}\} \subseteq \Omega_{n}$, consistent with our convention of identifying $g^{-1}(1)$ with $\tilde{g}^{-1}(1)$ when $g=\tilde{g}\circ \pi_{n}$.  In fact, when dealing with elements of $G_{p}^{(n)}$, we will usually consider them as functions with domain $\Omega_{n}$, and we will define subsets of $G_{p}^{(n)}$ in terms of the preimages $g^{-1}(1)$, which we identify with a subset of $\Omega_{n}$.

\subsection{Restrictions to cylinders}\label{secCylinders} Given $m, n\in \mathbb N$ with $m<n$, a string $\tau\in \Omega_{m}$, and an element $g\in G_{p}^{(n)}$, we define $g|_{\tau}$ to be the element of $G_{p}^{(n-m)}$ satisfying $g|_{\tau}(\omega) = g(\tau\omega)$ for all $\omega\in \{0,1\}^{(n-m)}$, where $\tau\omega\in \Omega$ is the concatenation of $\tau$ and $\omega$.  To give an explicit example: for $n=5$, $m=2$, and $g:\Omega\to \mathbb F_{7}$, if $\tau = 01$ and $\omega\in \Omega$, then $g|_{\tau}(\omega) = g(01\omega_{1}\omega_{2}\omega_{3}\dots)$.  With this notation we can identify $g\in G_{p}^{(n)}$ with the function $f:\Omega_{m}\to G_{p}^{(n-m)}$, where $f(\tau):=g|_{\tau}$ for each $\tau \in \Omega_{m}$.  This identification is used in Definition \ref{definitionStep}, where the sets $A$ of Theorems \ref{theoremMainDensity} and \ref{theoremPWsynd} are defined.

%If $g\in G_{p}$ and $\tau \in \Omega_{n}$, we define $g|_{\tau}\in G_{p}$ by $g|_{\tau}(\omega) = g(\tau\omega)$, where $\tau\omega\in \Omega$ is the concatenation of $\tau$ and $\omega$: $(\tau\omega)_{r} = \tau_{r} $ if $r\leq n$, and $(\tau\omega)_{r}=\omega_{r-n}$ if $r>n$.

\subsection{Upper Banach density in \texorpdfstring{$G_{p}$}{Gp}}  Since $G_{p}=\bigcup_{n=1}^{\infty} G_{p}^{(n)}$, the sequence $(G_{p}^{(n)})_{n\in \mathbb N}$ is a F{\o}lner sequence for $G_{p}$, as mentioned in Section \ref{sectionUBD}.  Consequently,  we have $d^{*}(A) \geq \limsup_{n\to \infty} \frac{|A\cap G_{p}^{(n)}|}{|G_{p}^{(n)}|}$ for every $A\subseteq G_{p}$.

\subsection{Hamming Balls}\label{sectionHammingBalls} For $n, k \in \mathbb N\cup\{0\}$, let $U(n,k)$ be the set of $g\in G_{p}^{(n)}$ satisfying $|\{\omega \in \Omega: g(\omega) \neq 0\}|_{n}\leq k$ (cf. Definition \ref{defContent}). This is  the \emph{Hamming ball of scale $n$ and radius $k$ around} $\mathbf 0\in G_{p}^{(n)}$.  In other words, $U(n,k)$ is the set of $g\in G_{p}$ which are constant on the cylinder sets $[\tau]$ for $\tau\in \Omega_{n}$ and $g|_{\tau}=\mathbf 0$ for at least $|\Omega_{n}|-k$ such $\tau$.

For $g \in G_{p}$, let $V(n,k) := U(n,k) + \mathbf{1}$, so that
\[V(n,k) = \{g\in G_{p}^{(n)}: |\{\omega \in \Omega: g(\omega) \neq 1\}|_{n} \leq k\}.\]
We call $V(n,k)$ the \emph{Hamming ball of scale $n$ and radius $k$ around} $\mathbf 1$.

\begin{remark}
 We call the sets $U(n,k)$ and $V(n,k)$ ``Hamming balls" as we may identify elements of $G_{p}^{(n)}$ with strings of length $2^{n}$ from the alphabet $\mathbb F_{p}$.  With this identification $U(n,k)$ is the set of strings differing from the constant $0$ string in at most $k$ coordinates.
\end{remark}

\begin{definition}
  Let $k, n\in \mathbb N$, $\delta>0$.  A set $S\subseteq G_{p}^{(n)}$ is a
  \begin{enumerate}
    \item[$\bullet$] \emph{set of $\delta$-density recurrence} if for every $A\subseteq G_{p}^{(n)}$ having $|A|\geq \delta|G_{p}^{(n)}|$, $(A-A)\cap S\neq \varnothing$.

    \item[$\bullet$]  \emph{set of $k$-chromatic recurrence} if for all partitions $G_{p}^{(n)}=\bigcup_{j=1}^{k} A_{j}$, there is a $j$ such that $(A_{j}-A_{j})\cap S\neq \varnothing$. \end{enumerate}
\end{definition}

In $G_{2}^{(n)}$, the sets $V(n,k)$ have the following important properties.

\begin{enumerate}

\item[(V1)] For $\delta<1/2$ and $n$ much larger than $k$, $V(n,k)$ is not a set of $\delta$-density recurrence: when $n$ is very large compared to $k$, there are sets $A\subseteq G_{2}^{(n)}$ having $|A| > |G_{2}^{(n)}|   \bigl(\frac{1}{2}-\varepsilon\bigr)$, such that $(A-A)\cap V(n,k) = \varnothing$.  In fact
    \[A:=\bigl\{g\in G_{2}^{(n)}: |g^{-1}(1)|_{n} > \tfrac{1}{2}|\Omega_{n}|+k\bigr\}\]
     is such a set.

\item[(V2)] $V(n,k)$ is a set of $k$-chromatic recurrence: if $G_{2}^{(n)} = \bigcup_{j=1}^{k} A_{j}$, there is a $j$ such that $(A_{j}-A_{j})\cap V(n,k)\neq \varnothing$.

\item[(V3)] $V(n,k)$ is a set of density recurrence when $k$ is comparable to a fixed multiple of $n$: for fixed $c, \delta>0$, there exists $N$ such that $(A-A) \cap V(n,\lfloor cn \rfloor)\neq \varnothing$ whenever $n\geq N$ and $A\subseteq G_{2}^{(n)}$ has $|A|\geq \delta|G_{2}^{(n)}|$.

\end{enumerate}

Property (V2) is established in the proof of Lemma \ref{lemmaChromatic}. Property (V1) is proved in Lemmas \ref{lemmaBase} and \ref{lemmaBaseEstimate}. Property (V3), which we do not use in this paper, is a corollary of a theorem of Kleitman \cite{Kleitman}, as shown in \cite{ForrestPaper}.

Each of the constructions in \cite{Kriz}, \cite[Theorem 1.2]{McCutcheonAlexandria}, and \cite[Theorem 3.35]{McCutcheonBook}  prove that $(R_{5})\cnot\implies (R_{4})$  by finding sets $\tilde{V}(n,k) \subseteq \mathbb Z$ imitating $V(n,k)$ and exploiting Properties (V1) and (V2), taking a union $\bigcup_{i=1}^{\infty}\tilde{V}(n_{i},k_{i})$ to construct the desired example.   Similarly, the constructions of  \cite{ForrestPaper}, \cite{ForrestThesis}, (proving $(R_{4})\cnot\implies (R_{3})$), and \cite{GriesmerRRPD} (proving $(R_{4}^{\bullet}) \cnot\implies (R_{3})$) find sets $\tilde{V}(n,k)\subseteq \mathbb Z$ imitating $V(n,k)$ and exploiting Property (V3), taking a union of these to get the desired example.  Every known example distinguishing some pair of the properties $(R_{3}), (R_{4})$, $(R_{5})$, $(R_{6})$, or $(R_{3}^{\bullet}), (R_{4}^{\bullet})$, $(R_{5}^{\bullet})$, $(R_{6}^{\bullet})$ follows this rough outline.  It would be interesting to find, for example, a set $S\subseteq \mathbb Z$ satisfying $(R_{5})$ but not $(R_{4})$, which is not constructed in this way.  To be more specific, we pose the following question.

\begin{question}\label{questionHereditary}
  \begin{enumerate}
    \item[(i)] Is there a set $S\subseteq S_{3}:=\{n^{3} : n\in \mathbb N\}$ satisfying $(R_{5})$ but not $(R_{4})$?

    \item[(ii)]  Is there a set $S\subseteq S_{5/2}:=\{\lfloor n^{5/2} \rfloor: n\in \mathbb N\}$ satisfying $(R_{5}^{\bullet})$ but not $(R_{4})$?
 \end{enumerate}
\end{question}
We use $S_{3}$ and $S_{5/2}$ because it appears that the constructions of \cite{Kriz} and \cite{McCutcheonAlexandria,McCutcheonBook} cannot produce subsets of $S_{3}$ and $S_{5/2}$ which are sets of topological recurrence, so a significant modification of the technique seems to be necessary to provide an affirmative answer to either part of Question \ref{questionHereditary}.  On the other hand, it would be interesting to find a set of measurable recurrence with the property that every subset thereof which is a set of topological recurrence is also a set of measurable recurrence, so we formulate the following question, which is open for every group $G$.

\begin{question}\label{questionHereditary2}
  Fix a countable abelian group $G$.  Is there a set of measurable recurrence $S\subseteq G$ such that every subset of $S$ which is a set of topological recurrence is also a set of measurable recurrence?
\end{question}

A negative answer to either part of Question \ref{questionHereditary} would provide a positive answer to Question \ref{questionHereditary2}, as it is well known that $S_{3}$ is a set of measurable recurrence (by a theorem due to Furstenberg \cite{Furstenberg} and S\'{a}rk\H{o}zy \cite{Sarkozy}, independently) and $S_{5/2}$ satisfies $(R_{1})$ (see \cite{BKQW}, for example).

\section{Proof of Theorem \ref{theoremMainDensity}}\label{sectionProof}

In this section we construct the sets $S, A \subseteq G_{2}$ described in Theorem \ref{theoremMainDensity}.  We maintain the notation and conventions of Section \ref{sectionGroups}. The set $S$ will be a union of some of the $V(n,k)$ (defined in Section \ref{sectionHammingBalls}), so we begin by showing that the $V(n,k)$ satisfy a quantitative version of chromatic recurrence.

\subsection{Chromatic recurrence properties of the \texorpdfstring{$V(n,k)$}{Vnk}}

\begin{lemma}\label{lemmaChromatic}
  Let $(n_{i})_{i\in \mathbb N}, (k_{i})_{i\in \mathbb N}$ be increasing sequences of natural numbers and let $(g_{i})_{i\in \mathbb N}$ be a sequence of elements of $G_{2}$ with $g_{i}\in G_{2}^{(n_{i})}$ for each $i$. Then $\bigcup_{i\in \mathbb N} (g_{i}+U(n_{i},k_{i})
   )\setminus \{0\}$ is a set of chromatic recurrence, and therefore a set of topological recurrence.

Consequently, every translate of $S:=\bigcup_{i=1}^{\infty} V(n_{i},k_{i})$ is a set of topological recurrence in $G_{2}$, and in fact, for every $g\in G_{2}$, $(g+S)\setminus \{0\}$ is a set of topological recurrence.
\end{lemma}

As in \cite{ForrestThesis, Kriz, McCutcheonAlexandria, McCutcheonBook}, we prove Lemma \ref{lemmaChromatic} as a consequence of the following theorem of Lov\'asz \cite{Lovasz}.

\begin{theorem}\label{thmLovasz}
  Let $k, r\in \mathbb N$, and let $E$ be the set of $r$-element subsets of $\{1,\dots, 2r+k\}$.  If $E=\bigcup_{j=1}^{k}E_{j}$, there is a $j\leq k$ and a disjoint pair of elements $e_{1}, e_{2}\in E$ such that $e_{1}, e_{2}\in E_{j}$.
\end{theorem}
%\begin{remark}
%  Theorem \ref{thmLovasz} is sharp: the chromatic number of the Kneser graph $K_{2r+k,r}$ is exactly $k+2$, while for our purposes we need only a weaker version: the chromatic number of the Kneser graphs $KG_{2r+k,r}$ tends to $\infty$ as $k\to \infty$, independently of $r$.
%\end{remark}
We also need the following elementary lemma, which we prove as a very special case of the Poincar{\'e} Recurrence Theorem.
\begin{lemma}\label{lemmaPoincare}
  For $k< n \in \mathbb N$, if $G_{2}^{(n)} = \bigcup_{j=1}^{k} A_{j}$, then
  \[
  (A_{j}-A_{j})\cap U(n,2k+2) \notin \{\varnothing,\{0\}\}
  \]
  for some $j \leq k$.
\end{lemma}
\begin{proof}
For some $j$ we have $|A_{j}| \geq \frac{1}{k}|G_{2}^{(n)}|$; fix such a $j$.  Note that $U(n,k+1),U(n,2k+2)\subseteq G_{2}^{(n)}$, $U(n,2k+2)$ contains the difference set $U(n,k+1)-U(n,k+1)$, and $|U(n,k+1)| \geq k+1$.  The sets $A_{j}+u$, $u\in U(n,k+1)$ cannot all be mutually disjoint, since that would imply
\[
|G_{2}^{(n)}| \geq |U(n,k+1)|\cdot |A_{j}| \geq \tfrac{k+1}{k}|G_{2}^{(n)}|> |G_{2}^{(n)}|.
\]  Hence there exist $u_{1} \neq u_{2} \in U(n,k+1)$ such that $(A_{j}+u_{1})\cap (A_{j}+u_{2}) \neq \varnothing$, meaning there exist $a, b\in A_{j}$ such that $a-b = u_{1}-u_{2}$.  Since $u_{1}\neq u_{2}$ and $u_{1}-u_{2} \in U(n,2k+2)$, we have shown that $(A_{j}-A_{j})\cap U(n,2k+2)\notin \{\varnothing, \{0\}\}$.
\end{proof}

\begin{proof}[Proof of Lemma \ref{lemmaChromatic}]
We will prove the following.

\begin{claim}  Fix $k< n \in \mathbb N$.  If $g\in G_{2}^{(n)}$ and $G_{2}^{(n)} = \bigcup_{j=1}^{k} A_{j}$, then for some $j\leq k$ we have $(A_{j}-A_{j}) \cap (g+U(n,3k+3)) \notin \{\varnothing,\{0\}\}$.
\end{claim}

\begin{proof}[Proof of Claim] Fix $g\in G_{2}^{(n)}$, $k\in \mathbb N$, and a partition $G_{2}^{(n)}=\bigcup_{j=1}^{k} A_{j}$. Let $X_{1}=g^{-1}(1)$, so that $X_{1}\subseteq \pi_{n}^{-1}(\Omega_{n})\subset \Omega$. We consider two cases based on $|X_{1}|_{n}$ (defined in Section \ref{secCylinders}).

\smallskip

\noindent \textbf{Case 1.  $|X_{1}|_{n}\geq 2+k$.}  In this case write $|X_{1}|_{n} = 2r+k$ if $|X_{1}|_{n}-k$ is even and write $|X_{1}|_{n}=2r+k+1$ if $|X_{1}|_{n}-k$ is odd, where $r\in \mathbb N$.  Let
\[
  D: = \{h\in G_{2}^{(n)}: h^{-1}(1)\subseteq X_{1} \text{ and } |h^{-1}(1)|_{n} = r\}
\]
Identify $D$ with the $r$-element subsets of $\{\tau\in \Omega_{n}: [\tau]\subset X_{1}\}$, where the identification is given by $h\leftrightarrow \{\tau\in \Omega_{n}:h([\tau]) = \{1\} \}$.  For $1\leq j \leq k$, let $D_{j}=A_{j}\cap D$, so that $D=\bigcup_{j=1}^{k} D_{j}$.  Theorem \ref{thmLovasz} implies that some $D_{j}$ contains two disjoint elements of $D$, say $h_{1}$ and $h_{2}$.  The difference $h_{1}-h_{2}$ satisfies $(h_{1}-h_{2})(\omega) = 0$ for $\omega\notin X_{1}$, and $|(h_{1}-h_{2})^{-1}(1)|_{n}=2r$, meaning $(h_{1}-h_{2})([\tau]) \neq g([\tau])$ for at most $k+1$ values of $\tau \in \Omega_{n}$.  It follows that
\[
h_{1}-h_{2} \in (g+U(n,k+1)) \subseteq (g+U(n,3k+3)),
\] and $h_{1}-h_{2}\neq 0$.   Since $h_{1}, h_{2}\in D_{j} \subseteq A_{j}$, this concludes the proof of the Claim in Case 1.

\smallskip
\noindent  \textbf{Case 2. $|X_{1}|_{n}<2+k$.} In this case $g+U(n,3k+3)$ contains the Hamming ball $U(n,2k+2)$, which has the desired property, by Lemma \ref{lemmaPoincare}.  This completes the proof of the Claim.  \end{proof}

To complete the proof of Lemma \ref{lemmaChromatic}, let $(n_{i})_{i\in \mathbb N}, (k_{i})_{i\in \mathbb N}$, and $(g_{i})_{i\in \mathbb N}$ be as in the statement of the lemma.  Let $R:=\bigcup_{i=1}^{\infty} g_{i}+U(n_{i},k_{i})$, $k\in \mathbb N$, and let $G_{2} = \bigcup_{j=1}^{k} A_{j}$ be a partition of $G_{2}$. Fix $i\in \mathbb N$ so that $n_{i}, k_{i}\geq  3k+3$.   Let $A_{j}' = A_{j}\cap G_{2}^{(n_{i})}$ for each $j\leq k$.  By the Claim, there exists $j\leq k$ such that $(A_{j}'-A_{j}')\cap (g_{i}+U(n_{i},k_{i})) \notin \{\varnothing,\{0\}\}$.  Consequently $(A_{j}-A_{j}) \cap R \notin \{\varnothing, \{0\}\}$.  Since the partition of $G_{2}$ was arbitrary, we have shown that $R\setminus \{0\}$ is a set of chromatic recurrence, as desired.  Proposition \ref{propositionCorrespondence} then implies $R\setminus \{0\}$ is a set of topological recurrence.

To prove that every translate of $S:=\bigcup_{i=1}^{\infty} V(n_{i},k_{i})$ is a set of topological recurrence, let $g\in G_{2}$ and choose $j$ sufficiently large that $g\in G_{2}^{(n_{j})}$.  Then $g+S \supseteq \bigcup_{i=j}^{\infty} g+\mathbf{1}+U(n_{i},k_{i})$, which is a set of topological recurrence, by the preceding paragraph. \end{proof}

\begin{problem}\label{problemHamming}
  Let $p$ be an odd prime and let $V(n,k)\subseteq G_{p}$ be as defined in Section \ref{sectionHammingBalls}.  Prove or disprove:
   \begin{center} \label{star}$(*)$ \hspace{.05in}\begin{minipage}{.9\textwidth} If $n_{i}, k_{i}\to \infty$, then every translate of $S:=\bigcup_{i=1}^{\infty} V(n_{i},k_{i})$ is a set of topological recurrence. \end{minipage}\end{center}
\end{problem}
If the statement ($*$) is true for a given $p$, then the results of \cite{GriesmerBohr} imply $(R_{5}^{\bullet})\cnot\implies (R_{4})$ for $G_{p}$, and that there are sets $S, A\subseteq G_{p}$ having $d^{*}(A)>0$ and every translate of $S$ is a set of topological recurrence, while $S+A$ is not piecewise syndetic.  This would provide a negative answer to Part (iii) of Question \ref{questionMain} and both parts of Question \ref{questionSumsets} for $G=G_{p}$.  If the statement ($*$) is false for some odd prime $p$, the results of \cite{GriesmerBohr} provide an example showing that $(R_{6}^{\bullet})\cnot\implies (R_{5})$ for the corresponding $G_{p}$, giving a negative answer to Part (i) of Question \ref{questionMain}.

\subsection{Some dense subsets of \texorpdfstring{$G_{2}$}{G2}.}

Our second task is to construct the sets $A$ of Theorems \ref{theoremMainDensity} and \ref{theoremPWsynd}.  We will find, for each $\varepsilon>0$, a set $A$ such that $d^{*}(A)>\frac{1}{2}-\varepsilon$, while $(A-A)\cap S = \varnothing$ for some $S=\bigcup_{j=1}^{\infty} V(n_{j},k_{j})$ where $n_{j}\to \infty$, $k_{j}\to \infty$.  Exhibiting such $A$ and $S$ will prove Theorem \ref{theoremMainDensity}, as Lemma \ref{lemmaChromatic} implies $(B-B+g)\cap S\neq \varnothing$ whenever $B$ is piecewise syndetic and $g\in G_{2}$.  We will also show that $(A'-A')\cap S =\varnothing$, where $A'=A+S$, leading to a proof of Theorem \ref{theoremPWsynd}.

The following definition uses the notation $|\cdot|_{n}$ defined in Section \ref{secConvention}.

\begin{definition}\label{definitionBase} For $n, m \in \mathbb N$ and $i\in \mathbb Z/2\mathbb Z$, let
\[A_{i}(n,m):= \{g\in G_{2}^{(n)}: |g^{-1}(i)|_{n} > \tfrac{1}{2}|\Omega_{n}|+m\}.\]
So $A_{i}(n,m)$ is the set of functions $g: \Omega\to \mathbb Z/2\mathbb Z$ which are constant on the cylinder sets $[\tau]$ for $\tau \in \Omega_{n}$, and the number of $\tau\in \Omega_{n}$ such that $g([\tau]) = \{i\}$ is greater than $\tfrac{1}{2}|\Omega_{n}|+m$.
\end{definition}

\begin{remark}
  The $A_{i}(n,m)$ are examples of niveau sets, first explicitly used in additive combinatorics by Ruzsa in \cite{RuzsaComponents, RuzsaProgressions}; see \cite{WolfPopular} for exposition and an application.

  While the $A_{0}(n,m)$ are essentially the Hamming balls $U(n,k)$, where $k=\frac{1}{2}|\Omega_{n}|-m$, we do not treat them as such.  Instead we view the $A_{i}(n,m)$ as the base case of the inductive Definition \ref{definitionStep}.
\end{remark}

\begin{remark} Letting $Z(n,m) = G_{2}^{(n)}\setminus (A_{0}(n,m)\cup A_{1}(n,m))$, we have a partition of $G_{2}^{(n)}$ into three sets  (it is easy to check that $A_{0}(n,m)$ and $A_{1}(n,m)$ are disjoint).  When $n$ is very large compared to $m$, $|A_{i}(n,m)|$ is close to $\frac{1}{2}|G_{2}^{(n)}|$, while $(A_{i}(n,m)-A_{i}(n,m))\cap V(n,m) = \varnothing$, as we shall prove.  In light of these facts and Lemma \ref{lemmaChromatic} we have a natural candidate for the set $A$ in Theorem \ref{theoremMainDensity}, namely $\bigcup_{j=1}^{\infty} A_{1}(n_{j},m_{j})$, where the $n_{j}\to \infty$ rapidly and the $m_{j}\to \infty$ slowly.  But this choice of $A$ will not work: setting
\begin{align*}
A': = A_{1}(n_{1},m_{1}) \cup A_{2}(n_{2},m_{2}), && S':=  V(n_{1},m_{1}) \cup V(n_{2},m_{2}),
\end{align*}
the desired disjointness $(A'-A') \cap S' = \varnothing$ is not true in general.  Instead of sets such as $A'$, we use sets which are easily understood in terms of their translates by elements of $V(n_{1},m_{1})\cup V(n_{2},m_{2})$.  Considering elements $g\in G_{2}^{(n_{2})}$ as functions $g:\Omega_{n_{1}} \to G_{2}^{(n_{2}-n_{1})}$, we think of  $A_{i}(n_{2}-n_{1},m_{2})$ as playing the role of $i\in \mathbb Z/2\mathbb Z$, and we let $A_{i}((n_{1},m_{1}),(n_{2},m_{2}))$ be the set of $g\in G_{2}^{(n_{2})}$ satisfying $g|_{\tau}\in A_{i}(n_{2}-n_{1},m_{2})$ for greater than $m_{1}$ values of $\tau \in \Omega_{n_{1}}$.  As we shall see in Lemma \ref{lemmaProperties}, it is easy to understand $g+v$ for  $g\in A_{i}((n_{1},m_{1}),(n_{2},m_{2}))$ and $v\in V(n_{1},m_{1})\cup V(n_{2},m_{2})$.  If $n_{2}-n_{1}$ is very large compared to $m_{2}$, then $|A_{i}(n_{2}-n_{1},m_{2})|$ is very close to $\frac{1}{2}|G_{2}^{(n_{2}-n_{1})}|$, so estimating $|A_{i}((n_{1},m_{1}),(n_{2},m_{2}))|$ is not difficult.  Furthermore, this construction can be iterated to produce sets $A \subseteq G_{2}^{(n_{l})}$ whose translates by elements of $V(n_{1},m_{1}) \cup \dots \cup V(n_{l},m_{l})$ are easily understood.
\end{remark}

The following lemma summarizes the relevant properties of the sets $A_{i}(n,m)$.  Recall from Section \ref{secPresentation} that $\mb 1$ denotes the element $g\in G_{2}$ having $g(\omega)=1\in \mathbb Z/2\mathbb Z$ for all $\omega \in \Omega$.

\begin{lemma}\label{lemmaBase}
  For all $n,m,\in \mathbb N$, $k\in \mathbb N\cup \{0\}$, $k< m$, and $i\in \mathbb Z/2\mathbb Z$, we have

  \begin{enumerate}
    \item[(i)]  $A_{i}(n,m) + \mathbf 1 = A_{i+1}(n,m)$,
\smallskip
    \item[(ii)] $A_{i}(n,m) + U(n,k) \subseteq A_{i}(n,m-k)$,
\smallskip
    \item[(iii)] $A_{i}(n,m) + U(n,k)+\mathbf 1 \subseteq A_{i+1}(n,m-k)$,
    \smallskip
    \item[(iv)] $A_{i}(n,m) \cap A_{i+1}(n,m-k) = \varnothing$,
    \smallskip
    \item[(v)] If $m'< m$, then $A_{i}(n,m) \subseteq A_{i}(n,m')$.
  \end{enumerate}
\end{lemma}

\begin{proof}
  Part (i) follows from the definition of $A_{i}(n,m)$ and the fact that $(g+\mathbf{1})^{-1}(i) = g^{-1}(i+1)$ for all $g\in G_{2}$.

  To prove Part (ii), note that for each $g\in G_{2}^{(n)}$, $u\in U(n,k)$,  we have $g([\tau]) = (g+u)([\tau])$ for at least $|\Omega_{n}|-k$ values of $\tau\in \Omega_{n}$.  We then have $|(g+u)^{-1}(i)|_{n} \geq |g^{-1}(i)|_{n}-k$ for every such $g$, and Part (ii) now follows from the definition of $A(n,m)$.

  Part (iii) follows directly from parts (i) and (ii).

  To prove Part (iv) we exploit Observation \ref{obsContent}.  Assume, to get a contradiction, that $A_{i}(n,m) \cap A_{i+1}(n,m-k) \neq  \varnothing$, and let $g$ be an element of the intersection.  We then have $|g^{-1}(i)|_{n} \geq \frac{1}{2}|\Omega_{n}|+m$ and $|g^{-1}(i+1)|_{n}\geq \frac{1}{2}|\Omega_{n}| + m-k$.  Since $g^{-1}(i) \cap g^{-1}(i+1) = \varnothing$, we then have $|\Omega_{n}| \geq |\Omega_{n}|+2m-k$, a contradiction.

  Part (v) follows immediately from the definition of $A_{i}(n,m)$. \end{proof}

\begin{remark}
  The $A_{i}(n,m)$ are Hamming balls of radius $\frac{1}{2}|\Omega_{n}|-m-1$ around $i\mathbf{1}\in G_{2}^{(n)}$, so Parts (ii) and (iii) of Lemma \ref{lemmaBase} are consequences of the fact that $U(n,k)+U(n,k')= U(n,k+k')$.
\end{remark}

In the following definition we will use the restrictions $g|_{\tau}$, defined in Section \ref{secCylinders}.

\begin{definition}\label{definitionStep} For all  $l\in \mathbb N$, $i\in \mathbb Z/2\mathbb Z$, and sequences of $l$ pairs of natural numbers $\mathbf n_{l}=(n_{1},m_{1}),\dots, (n_{l},m_{l})$ where $n_{1}< \dots < n_{l}$, we will define a set $A_{i}(\mathbf n_{l}) \subseteq G_{2}^{(n_{l})}$.  For $l=1$, $A_{i}(n_{1},m_{1})$ is defined in Definition \ref{definitionBase}. For $l\geq 2$, we  define $A_{i}(\mathbf n_{l})$ inductively, assuming $A_{i}((n_{1}',m_{1}'), \dots, (n_{l-1}',m_{l-1}'))$ is defined for every sequence of pairs where $n_{1}' < \dots < n_{l-1}' $. In particular, $A_{i}(\bar{\bn}_{l-1})$ is defined, where $\bar{\bn}_{l-1}:=(n_{2}-n_{1},m_{2}),\dots,(n_{l}-n_{1},m_{l})$.  For each $i\in \mathbb Z/2\mathbb Z$, define $A_{i}((n_{1},m_{1}),\dots, (n_{l},m_{l}))$ to be the set of $g\in G_{2}^{(n_{l})}$ such that
\begin{align}
\label{eqnBias}  |\{\tau \in \Omega_{n_{1}}: g|_{\tau} \in A_{i}(\bar{\bn}_{l-1})\}| > \tfrac{1}{2}|\Omega_{n_{1}}|+m_{1}.
\end{align}
\end{definition}
For example, $A_{1}((10,3),(50,7))$ is the set of $g\in G_{2}^{(50)}$ such that $g|_{\tau} \in A_{1}(40,7)$ for  $> \frac{1}{2}|\Omega_{10}|+3$ values of $\tau \in \Omega_{10}$.

\begin{remark}
The proofs of Theorems \ref{theoremMainDensity} and \ref{theoremPWsynd} are a straightforward exploitation of Definition \ref{definitionStep}. The remaining proofs in this subsection are tedious due to the inductive nature of the definition.
\end{remark}

We adopt the following conventions in the sequel.
\begin{enumerate}
  \item[$\bullet$] The symbol $\bn_{l}$ abbreviates the symbol $(n_{1},m_{1}),\dots,(n_{l},m_{l})$.

\item[$\bullet$] The symbol $\bar{\bn}_{l-1}$ abbreviates $(n_{2}-n_{1},m_{2}),\dots,(n_{l}-n_{1},m_{l})$.
 \end{enumerate}
\begin{lemma}\label{lemmaProperties} For all $l\in \mathbb N$,  sequences $\mathbf n_{l}= (n_{1},m_{1}),\dots, (n_{l},m_{l})$ with $n_{j}, m_{j}\in \mathbb N$, $n_{1}<\dots < n_{l}$, $i\in \mathbb Z/2\mathbb Z$, $j\in \{1,\dots, l\}$, $k_{j} \in \mathbb N\cup \{0\}$, $k_{j}< m_{j}$ we have

\begin{enumerate}
  \item[(i)] $A_{i}(\bn_{l})+\mathbf 1 = A_{i+1}(\bn_{l})$,
\smallskip
\item[(ii)] if $u\in U(n_{j},k_{j})$, then
\[A_{i}(\bn_{l})+u \subseteq A_{i}((n_{1},m_{1}), \dots, (n_{j},m_{j}-k_{j}), \dots, (n_{l},m_{l})),\]

\item[(iii)]  if $u\in U(n_{j},k_{j})$, then
\begin{align*}
  A_{i}(\bn_{l})+u+\mathbf 1 \subseteq A_{i+1}((n_{1},m_{1}), \dots, (n_{j},m_{j}-k_{j}), \dots, (n_{l},m_{l})),
\end{align*}

\item[(iv)]  the sets $A_{i}(\bn_{l})$ and
\begin{align*}
   A_{i+1}':=A_{i+1}((n_{1},m_{1}), \dots, (n_{j},m_{j}-k_{j}), \dots, (n_{l},m_{l}))
\end{align*}
are disjoint.   In particular $A_{0}(\bn_{l})\cap A_{1}(\bn_{l}) = \varnothing$.

\item[(v)]  If $m_{j}' \leq m_{j}$ then
\[A_{i}(\bn_{l})\subseteq A_{i}((n_{1},m_{1}),\dots,(n_{j},m_{j}'),\dots, (n_{l},m_{l})).\]

\end{enumerate}

\end{lemma}

\begin{proof}
We prove each of these statements by induction on $l$.  For each of Parts (i)-(v), the base case of the induction is $l=1$, which is the corresponding part of Lemma \ref{lemmaBase}.  We now fix $l\in \mathbb N$, $l>1$.  For the induction hypothesis, we assume each of (i) - (v) holds for all sequences $(n_{1}',m_{1}'),\dots, (n_{l-1}',m_{l-1}')$, where $n_{j}'$, $m_{j}' \in \mathbb N$ and $n_{1}' < \dots < n_{l-1}'$.  In particular, given a sequence $\bn_{l}=(n_{1},m_{1}),\dots, (n_{l},m_{l})$ of length $l$, each of (i)-(v) holds for the sequence $\bar{\bn}_{l-1}= (n_{2}-n_{1},m_{2}),\dots, (n_{l}-n_{1}, m_{l})$ of length $l-1$.

To prove (i), let $g\in A_{i}(\bn_{l})$, so that $g|_{\tau} \in A_{i}(\bar{\bn}_{l-1})$ for greater than $\frac{1}{2}|\Omega_{n_{1}}|+m_{1}$ values of $\tau\in \Omega_{n_{1}}$.  For each such $\tau$, $(g+\mathbf{1})|_{\tau} \in A_{i+1}(\bar{\bn}_{l-1})$, by the induction hypothesis.  Then $(g+\mathbf{1})|_{\tau} \in A_{i+1}(\bar{\bn}_{l-1})$ for greater than $\frac{1}{2}|\Omega_{n_{1}}|+m_{1}$ values of $\tau \in \Omega_{n_{1}}$, so $g+\mathbf{1} \in A_{i+1}(\bn_{l})$ by definition.  By symmetry we conclude that if $g\in A_{i+1}(\bn_{l})$, then $g-\mathbf{1}\in A_{i}(\bn_{l})$, and we conclude that $A_{i}(\bn_{l})+\mathbf{1}=A_{i+1}(\bn_{l})$.

To prove (ii) we consider two cases.

\noindent \textbf{Case 1:  $j=1$.}  Let $g\in A_{i}(\bn_{l})$ and $u\in U(n_{1},k_{1})$.  Then $g|_{\tau}\in A_{i}(\bar{\bn}_{l-1})$ for greater than $m_{1}$ values of $\tau\in \Omega_{n_{1}}$, while $(g+u)|_{\tau} = g|_{\tau}$ for all but $k_{1}$ values of $\tau\in \Omega_{n_{1}}$.  It follows that \[
|\{\tau \in \Omega_{n_{1}}: (g+u)|_{\tau} \in A_{i}(\bar{\bn}_{l-1})\}|>m_{1}-k_{1},
\] so $g+u \in A_{i}((n_{1},m_{1}-k_{1}),\dots, (n_{l},m_{l}))$, by definition.

\noindent \textbf{Case 2: $2 \leq j \leq l$.}  Again let $g\in A_{i}(\bn_{l})$ and $u\in U(n_{j},k_{j})$. We have $u|_{\tau} \in U(n_{j}-n_{1}, k_{j})$ for each $\tau \in \Omega_{n_{1}}$.  Then for each such $\tau$ where $g|_{\tau} \in A_{i}(\bar{\bn}_{l-1})$, the induction hypothesis implies
\[(g+u)|_{\tau} \in A_{i}((n_{2}-n_{1},m_{2}), \dots, (n_{j}-n_{1}, m_{j}-k_{j}), \dots, (n_{l}-n_{1},m_{l})).\]  The above inclusion then occurs for $>\frac{1}{2}|\Omega_{n_{1}}|+m_{1}$ values of $\tau\in \Omega_{n_{1}}$, so $g+u \in A_{i}((n_{1},m_{1}),\dots, (n_{j}, m_{j}-k_{j}), \dots, (n_{l},m_{l}))$.  This completes the proof of Part (ii).

Part (iii) follows immediately from Parts (i) and (ii).

We prove Part (iv) by induction on $l$.  The base case of the induction is $l=1$, which is Part (iv) of Lemma \ref{lemmaBase}. For the induction step we set $i=1$, as the case $i=0$ follows by symmetry.  Let $l>1$ and assume, to get a contradiction, that $A_{1}(\bn_{l})$ and $A_{0}'$ are not disjoint, and let $g\in A_{1}(\bn_{l}) \cap A_{0}'$.   We now consider two cases.

\noindent \textbf{Case 1: $j=1$.} The induction hypothesis implies $A_{0}(\bar{\bn}_{l-1})$ and $A_{1}(\bar{\bn}_{l-1})$ are disjoint.   We have
\begin{align}
\label{eqnOver1}  |\{\tau \in \Omega_{n_{1}}: g|_{\tau} \in A_{1}(\bar{\bn}_{l-1})\}| &> \tfrac{1}{2}|\Omega_{n_{1}}| + m_{1},  \text{ since } g\in A_{1}(\bn_{l}),\\
\label{eqnOver2}  |\{\tau \in \Omega_{n_{1}}: g|_{\tau} \in A_{0}(\bar{\bn}_{l-1})\}| &> \tfrac{1}{2}|\Omega_{n_{1}}| + m_{1}-k_{1},  \text{ since } g\in A_{0}'.
\end{align}
Inequalities (\ref{eqnOver1}) and (\ref{eqnOver2}) together imply $|\Omega_{n_{1}}| > |\Omega_{n_{1}}| + 2 m_{1}-k_{1}$,  contradicting the assumption $k_{1}\leq m_{1}$.

\noindent \textbf{Case 2: $j>1$.} The induction hypothesis implies that the sets $A_{1}(\bar{\bn}_{l-1})$ and
\[
A_{0}'':= A_{0}((n_{2}-n_{1},m_{2}),\dots, (n_{j}-n_{1},m_{j}-k_{j}), \dots, (n_{l}-n_{1},m_{l}))
\]
are disjoint.  Then
\begin{align}
\label{eqnOver3}  |\{\tau \in \Omega_{n_{1}}: g|_{\tau} \in A_{1}(\bar{\bn}_{l-1})\}| &> \tfrac{1}{2}|\Omega_{n_{1}}| + m_{1},  && \text{since } g\in A_{1}(\bn_{l}),\\
\label{eqnOver4}  |\{\tau \in \Omega_{n_{1}}: g|_{\tau} \in A_{0}''\}| &> \tfrac{1}{2}|\Omega_{n_{1}}| + m_{1},  && \text{since } g\in A_{0}'.
\end{align}
Inequalities (\ref{eqnOver3}) and (\ref{eqnOver4}) together imply $|\Omega_{n_{1}}| > |\Omega_{n_{1}}| + 2 m_{1}$, a contradiction.  This completes the proof of Part (iv).

For Part (v) we again use induction on $l$.  The base case, $l=1$, is Part (v) of Lemma \ref{lemmaBase}, so we establish the induction step.  Assume $l>1$.  The induction hypothesis implies
\[
A_{i}(\bar{\bn}_{l-1}) \subseteq A_{i}((n_{2}-n_{1},m_{2}),\dots (n_{j}-n_{1},m_{j}'),\dots, (n_{l}-n_{1},m_{l})),
\]
and the definition of $A_{i}(\bn_{l})$ then implies the conclusion.  This completes the proof of Part (v) and the proof of the Lemma.  \end{proof}

% \begin{corollary}
%   If $1\leq j \leq l$ and $k_{j}\leq m_{j}$, then $(A_{i}(\bn_{l})-A_{i}(\bn_{l}))\cap V(n_{j},k_{j}) = \varnothing$.
% \end{corollary}
%
% \begin{proof}
%   Note that $(A_{i}(\bn_{l})-A_{i}(\bn_{l}))\cap V(n_{j},k_{j}) = \varnothing$ if and only if $A_{i}(\bn_{l})\cap (V(n_{j},k_{j})+A_{i}(\bn_{l})) = \varnothing$.  The last equation follows from Parts (iii) and (iv) of Lemma \ref{lemmaProperties}.
% \end{proof}

\begin{lemma}\label{lemmaMonotone}
  For all $i\in \mathbb Z/2\mathbb Z$ and all $\bn_{l}=(n_{1},m_{1}),\dots,(n_{l},m_{l})$, $\bn_{l+1} = (n_{1},m_{1}),\dots,(n_{l+1},m_{l+1})$ where $n_{1}<\dots<n_{l+1}$, $n_{j}, m_{j}\in \mathbb N$, $A_{i}(\bn_{l}) \subseteq A_{i}(\bn_{l+1})$.
\end{lemma}

\begin{proof}
We consider the case $i=1$.  The case $i=0$ follows by symmetry.

 We proceed by induction on $l$.  The base case is the containment $A_{1}(n_{1},m_{1}) \subseteq A_{1}((n_{1},m_{1}),(n_{2},m_{2}))$.  Let $g\in A_{1}(n_{1},m_{1})$.  We must show that $g|_{\tau} \in A_{1}(n_{2}-n_{1},m_{2})$ for $>m_{1}$ values of $\tau \in \Omega_{n_{1}}$.  In fact $g|_{\tau} = \mathbf{1} \in A_{1}(n_{2}-n_{1},m_{2})$ for  $>m_{1}$ values of $\tau\in \Omega_{n_{1}}$, so we are done with the base case.

Now assume $l>1$.  Let $g\in A_{1}(\bn_{l})$.  The induction hypothesis implies $A_{1}(\bar{\bn}_{l-1}) \subseteq A_{1}(\bar{\bn}_{l})$, where $\bar{\mathbf{n}}_{l} =(n_{2}-n_{1},m_{2}),\dots,(n_{l+1}-n_{1},m_{l+1})$.  We must show that
\begin{align}
\label{eqnMonotone}   g|_{\tau} \in A_{1}((n_{2}-n_{1},m_{2}),\dots, (n_{l+1}-n_{1},m_{l+1})) =: A_{1}(\bar{\bn}_{l})
 \end{align} for $> m_{1}$ values of $\tau\in \Omega_{n_{1}}$.  By the definition of $A_{1}(\bn_{l})$, we have $g|_{\tau} \in A_{1}(\bar{\bn}_{l-1})$ for  $> \frac{1}{2}|\Omega_{n_{1}}|+m_{1}$ values of $\tau \in \Omega_{n_{1}}$, so the induction hypothesis implies that the inclusion (\ref{eqnMonotone}) holds for the required number of $\tau\in \Omega_{n_{1}}$.
\end{proof}

\subsection{Constructing elements of \texorpdfstring{$A_{i}(\bn_{l})$}{Ainl}.}

In this subsection we estimate the cardinality of $|A_{i}(n,m)|$ and construct some elements of $A_{i}(\bn_{l})$, for the purpose of estimating $|A_{i}(\bn_{l})|$ in the next subsection.

\begin{lemma}\label{lemmaBaseEstimate}
  Let $m\in \mathbb N$.  Then $\lim_{n\to \infty} \frac{|A_{i}(n,m)|}{|G_{2}^{(n)}|} = \frac{1}{2}$ for all $i\in \mathbb Z/2\mathbb Z$.
\end{lemma}

\begin{remark}We can also write the conclusion of Lemma \ref{lemmaBaseEstimate} as
\begin{equation}\label{baseLittleO}
  |A_{i}(n,m)| = |G_{2}^{(n)}|\bigl(\tfrac{1}{2}+o(1)\bigr),
\end{equation}
where $o(1)$ is a quantity tending to $0$ as $n\to \infty$ and $m$ remains fixed.\end{remark}

\begin{proof}
  Let $Z(n,m):= G_{2}^{(n)} \setminus (A_{0}(n,m)\cup A_{1}(n,m))$. Lemma \ref{lemmaBase} implies $|A_{0}(n,m)|= |A_{1}(n,m)|$ and $A_{0}(n,m)\cap A_{1}(n,m)=\varnothing$, so it suffices to show that
  \begin{align}\label{eqnZlim}
    \lim_{n\to \infty} \frac{|Z(n,m)|}{|G_{2}^{(n)}|} = 0.
  \end{align}
Now $Z(n,m)=\{g\in G_{2}^{(n)}:  \frac{1}{2}|\Omega_{n}| - m \leq |g^{-1}(1)|_{n} \leq \frac{1}{2}|\Omega_{n}| + m \}$, so
\begin{align*}
  |Z(n,m)| &\leq (2m+1)\max_{k\in \{0,\dots, |\Omega_{n}|\}} \binom{|\Omega_{n}|}{k}\\
  &= (2m+1)\binom{|\Omega_{n}|}{\lfloor |\Omega_{n}|/2\rfloor}.
\end{align*}
Since $|G_{2}^{(n)}| = 2^{|\Omega_{n}|}$ and $\binom{|\Omega_{n}|}{\lfloor |\Omega_{n}|/2\rfloor} = o(2^{|\Omega_{n}|})$, we have established Equation (\ref{eqnZlim}).
\end{proof}

To estimate the cardinality of $A_{i}(\bn_{l})$, it helps to have a ``bottom up" inductive construction, somewhat dual to the ``top down" inductive definition given in \ref{definitionStep}.

\begin{definition}Given $g\in G_{2}^{(n_{l-1})}$ and $m_{l}\in \mathbb N$, let $G_{2}^{(n_{l})}[g,m_{l}]$ be the set of $h\in G_{2}^{(n_{l})}$ satisfying
\begin{equation}\label{equationAdjoin}
  h|_{\tau}\in A_{g(\tau)}(n_{l}-n_{l-1},m_{l}) \text{ for all } \tau \in \Omega_{n_{l-1}},
\end{equation}
where we abuse notation and use $g(\tau)$ to denote $i\in \mathbb Z/2\mathbb Z$ if $g([\tau])=\{i\}$.
\end{definition}

\begin{lemma}\label{lemmaBottom}  Let $l\in \mathbb N$, $l>1$.  Let $n_{1}<\dots< n_{l}\in \mathbb N$, $m_{j}\in \mathbb N$,
$i\in \mathbb Z/2\mathbb Z$ and $g\neq g' \in A_{i}((n_{1},m_{1}),\dots, (n_{l-1},m_{l-1}))$.  Then
  \begin{enumerate}
    \item[(i)] $G_{2}^{(n_{l})}[g,m_{l}] \subseteq A_{i}((n_{1},m_{1}),\dots,(n_{l},m_{l}))$,

\smallskip

    \item[(ii)]  $G_{2}^{(n_{l})}[g,m_{l}] \cap G_{2}^{(n_{l})}[g',m_{l}] = \varnothing$.

\smallskip

    \item[(iii)] If $g\in G_{2}^{(n_{l-1})}$ and $m_{l}\in \mathbb N$ are fixed, then
\begin{equation}\label{equationEndSize}
          |G_{2}^{(n_{l})}[g,m_{l}]| = |G_{2}^{(n_{l})}|\Bigl(\frac{1}{|G_{2}^{(n_{l-1})}|} + o(1)\Bigr),
\end{equation}
where $o(1)$ is a quantity tending to $0$ as $n_{l}\to \infty$.
     \end{enumerate}
\end{lemma}
\begin{proof}
  We prove Part (i) by induction on $l$.  For $l=1$ there is nothing to prove, so the base case is $l=2$.    We set $i=1$, as the case $i=0$ follows by symmetry. Fix $g\in A_{1}(n_{1},m_{1})$.  We must prove that $G_{2}^{(n_{2})}[g,m_{2}] \subseteq A_{1}((n_{1},m_{1}),(n_{2},m_{2}))$.  Let $h\in G_{2}^{(n_{2})}[g, m_{2}]$, so that for $\tau \in \Omega_{n_{1}} $, $h|_{\tau} \in A_{1}(n_{2}-n_{1},m_{2})$ whenever $g|_{\tau} = \mathbf 1$, which occurs for at least $\frac{1}{2}|\Omega_{n_{1}}| + m_{1}$ values of $\tau \in \Omega_{n_{1}}$.  Then $h\in A_{1}((n_{1},m_{1}),(n_{2},m_{2}))$, by definition.

  Now fix $l\in \mathbb N$, $l>2$.  The induction hypothesis implies that if $n_{1}'<\dots < n_{l-1}'$, $m_{j}'\in \mathbb N$, and $g\in A_{1}((n_{1}',m_{1}'),\dots,(n_{l-2}',m_{l-2}'))$, then
  \[G_{2}^{(n_{l-1}')}[g,m_{l-1}']\in A_{1}((n_{1}',m_{1}'),\dots,(n_{l-1}',m_{l-1}')).\]
Let $g\in A_{1}((n_{1},m_{1}),\dots, (n_{l-1},m_{l-1}))$, and let $h\in G_{2}^{(n_{l})}[g,m_{l}]$, so that $h|_{\tau} \in A_{1}((n_{l}-n_{l-1},m_{l}))$ whenever $\tau\in \Omega_{n_{l-1}}$ and $g|_{\tau} = \mathbf 1$.  For all $\psi \in \Omega_{n_{1}}$, we then have $h|_{\psi} \in G_{2}^{(n_{l}-n_{1})}[g|_{\psi},m_{l}]$.  By the induction hypothesis, the inclusion
\begin{equation}\label{equationInclusion}
   g|_{\psi}\in A_{1}((n_{2}-n_{1},m_{2}),\dots, (n_{l-1}-n_{1},m_{l-1}))
 \end{equation}implies $h|_{\psi} \in A_{1}((n_{2}-n_{1},m_{2}),\dots, (n_{l}-n_{1},m_{l}))$.  Since the inclusion (\ref{equationInclusion}) occurs for at least $\frac{1}{2}|\Omega_{n_{1}}|+m_{1}$ values of $\psi\in \Omega_{n_{1}}$, we then have $h\in A_{1}((n_{1},m_{1}),\dots, (n_{l},m_{l}))$.  This completes the induction and the proof of Part (i).

To prove Part (ii) we use Part (iv) of Lemma \ref{lemmaBase}.  Let $h\in G_{2}^{(n_{l})}[g,m_{l}]$, $h'\in G_{2}^{(n_{l})}[g',m_{l}]$.  Since $g\neq g'$, there exists $\tau \in \Omega_{n_{l-1}}$ such that $g|_{\tau} \neq g'|_{\tau}$.  Then for some $i\in \mathbb Z/2\mathbb Z$, $h|_{\tau} \in A_{i}(n_{l}-n_{l-1},m_{l})$ while $h'|_{\tau} \in A_{i+1}(n_{l}-n_{l-1},m_{l})$. Then $h|_{\tau} \neq h'|_{\tau}$ by disjointness of $A_{i}(n_{l}-n_{l-1},m_{l})$, $i=0,1$ (Part (iv) of Lemma \ref{lemmaBase}).  It follows that $h\neq h'$, as desired.

For Part (iii), we identify $G_{2}^{(n_{l})}[g,m_{l}]$ with the set of functions $h:\Omega_{n_{l-1}} \to G_{2}^{(n_{l}-n_{l-1})}$ satisfying $h(\tau)\in A_{g(\tau)}(n_{l}-n_{l-1},m_{l})$ for $\tau \in \Omega_{n_{l-1}}$.  There are
\[
\prod_{\tau \in \Omega_{n_{l-1}}} |A_{g(\tau)}(n_{l}-n_{l-1},m_{l})|
\]
such functions, so the estimate (\ref{baseLittleO}) implies

\begin{align*}
  |G_{2}^{(n_{l})}[g,m_{l}]| &= |G_{2}^{(n_{l}-n_{l-1})}|^{|\Omega_{n_{l-1}}|}\bigl(\tfrac{1}{2}+o(1)\bigr)^{|\Omega_{n_{l-1}}|}\\
  &= |G_{2}^{(n_{l})}|\Bigl(\frac{1}{2^{|\Omega_{n_{l-1}}|}} + o(1)\Bigr)\\
  &= |G_{2}^{(n_{l})}|\Bigl(\frac{1}{|G_{2}^{(n_{l-1})}|} + o(1)\Bigr),
\end{align*}
as desired.
\end{proof}

\subsection{Estimating \texorpdfstring{$|A_{i}(\bn_{l})|$}{|A(nl)|}.}

\begin{lemma}\label{lemmaEstimate}
  Fix $n_{1}<  \dots < n_{l-1}\in \mathbb N$ and $m_{1}, \dots, m_{l}\in \mathbb N$.  Then for all $i\in \mathbb Z/2\mathbb Z$
  \begin{align*}
    \liminf_{n_{l}\to \infty} \frac{|A_{i}((n_{1},m_{1}),\dots,(n_{l},m_{l}))|}{|G_{2}^{(n_{l})}|} \geq \frac{|A_{i}((n_{1},m_{1}),\dots,(n_{l-1},m_{l-1}))|}{|G_{2}^{(n_{l-1})}|}.
  \end{align*}
\end{lemma}

\begin{proof}  In this proof, $o(1)$ denotes a quantity which tends to $0$ as $n_{l}\to \infty$. We estimate the cardinality of $A_{1}(\bn_{l}):=A_{1}((n_{1},m_{1}),\dots,(n_{l},m_{l}))$, since Lemma \ref{lemmaProperties} implies $|A_{1}(\bn_{l})|=|A_{0}(\bn_{l})|$.

  For $g\in G_{2}^{(n_{l-1})}$, consider $B(g):=G_{2}^{(n_{l})}[g,m_{l}]$.
Lemma \ref{lemmaBottom} says that the $B(g)$ are mutually disjoint and $\bigcup_{g\in A_{1}(\bn_{l-1})} B(g) \subseteq A_{1}(\bn_{l})$, where $A_{1}(\bn_{l-1}):=A_{1}((n_{1},m_{1}),\dots, (n_{l-1},m_{l-1}))$. Part (iii) of Lemma \ref{lemmaBottom} then implies
\begin{align*}
  |A_{1}(\bn_{l})| &\geq |A_{1}(\bn_{l-1})| \min_{g\in G_{2}^{(n_{l-1})}} |B(g)|\\
  &= |A_{1}(\bn_{l-1})| |G_{2}^{(n_{l})}|\Bigl(\frac{1}{|G_{2}^{(n_{l-1})}|}+ o(1)\Bigr) && \text{by (\ref{equationEndSize})}\\
  &= |G_{2}^{(n_{l})}| \Bigl(\frac{|A_{1}(\bn_{l-1})| }{|G_{2}^{(n_{l-1})}|}+o(1)\Bigr),
\end{align*}
which implies the conclusion of the lemma. \end{proof}

\subsection{Proofs of Theorems.}
In this section we prove Theorems \ref{theoremMainDynamic}, \ref{theoremMainDensity}, \ref{theoremPreciseDynamic}, and \ref{theoremPWsynd}.  The set $S$ we construct for these theorems will be a union of some of the $V(n,k)$ defined in Section \ref{sectionHammingBalls}.

\begin{lemma}\label{lemmaPush}
  Let $(n_{j})_{j\in \mathbb N}$ be an increasing sequence of natural numbers, and let $(m_{j})_{j\in \mathbb N}$, $(k_{j})_{j\in \mathbb N}$ be sequences of natural numbers.  Let $S=\bigcup_{j=1}^{\infty} V(n_{j},k_{j})$, $i\in \mathbb Z/2\mathbb Z$, and $A= \bigcup_{l=1}^{\infty} A_{i}((n_{1},m_{1}),\dots,(n_{l},m_{l}))$.  Then
  \[S+A \subseteq \bigcup_{l=1}^{\infty} A_{i+1}((n_{1},m_{1}-k_{1}),\dots, (n_{l},m_{l}-k_{l})).\]
\end{lemma}

\begin{proof}
As usual $\bn_{l}$ abbreviates the expression $(n_{1},m_{1}),\dots,(n_{l},m_{l})$.  Since $S+A = \bigcup_{l, j = 1}^{\infty} V(n_{j},k_{j}) + A_{i}(\bn_{l})$, it suffices to show that for each $j, l$,
  \begin{align}\label{eqnPart}
  V(n_{j},k_{j}) + A_{i}(\bn_{l}) \subseteq A_{i+1}((n_{1},m_{1}-k_{1}),\dots, (n_{r},m_{r}-k_{r})),
  \end{align} where $r=\max(l,j)$.  When $j\leq l$, the containment follows from Parts (iii) and (v) of Lemma \ref{lemmaProperties}.  When $j> l$, Lemma \ref{lemmaMonotone} implies $A_{i}(\bn_{l}) \subseteq A_{i}(\bn_{j})$, so that $V(n_{j},k_{j})+A_{i}(\bn_{l}) \subseteq V(n_{j},k_{j})+A_{i}(\bn_{j})$, and again Parts (iii) and (v) of Lemma \ref{lemmaProperties} imply the containment (\ref{eqnPart}). \end{proof}

\begin{proof}[Proof of Theorems \ref{theoremMainDensity} and \ref{theoremPWsynd}]
  Let $\varepsilon>0$ and let $(k_{j})_{j\in \mathbb N}$ be an increasing sequence of natural numbers.  Let $m_{j}= 3k_{j}$ for each $j$.  By Lemmas \ref{lemmaBaseEstimate} and \ref{lemmaEstimate}, we may choose an increasing sequence $(n_{j})_{j\in \mathbb N}$ of natural numbers such that
  \begin{align}\label{eqnUBD}
    \liminf_{l\to \infty} \frac{|A_{1}((n_{1},m_{1}),\dots,(n_{l},m_{l}))|}{|G_{2}^{(n_{l})}|} \geq \frac{1}{2}-\varepsilon.
  \end{align}
Let $A = \bigcup_{l=1}^{\infty} A_{1}(\bn_{l})$ and let $S=\bigcup_{l=1}^{\infty} V(n_{l},k_{l})$.  Inequality (\ref{eqnUBD}) implies $d^{*}(A)\geq \frac{1}{2}-\varepsilon$ (see Section \ref{sectionUBD}), and Lemma \ref{lemmaChromatic} implies every translate of $S$ is a set of chromatic recurrence.  By Lemma \ref{lemmaPWrecurrence}, we then have that $(C-C)\cap S \neq \varnothing$ whenever $C$ is piecewise syndetic.

We will show that
\begin{align}
 \label{equationDiff1} (A-A)\cap S &= \varnothing,\\
 \label{equationDiff2} [(A+S)-(A+S)]\cap S &=\varnothing.
\end{align}
Equation (\ref{equationDiff1}) will complete the proof of Theorem \ref{theoremMainDensity}, as Lemma \ref{lemmaPWrecurrence} implies that the existence of $g\in G$ and a piecewise syndetic $B\subseteq G$ such that $g+(B-B)\subseteq A-A$ would contradict Equation (\ref{equationDiff1}) and the fact that every translate of $S$ is a set of chromatic recurrence.

 Equation (\ref{equationDiff2}) implies that $A+S$ is not piecewise syndetic, since $S$ has nonempty intersection with $C-C$ whenever $C$ is piecewise syndetic.

Note that Equation (\ref{equationDiff1}) implies $A\cap (A+\mathbf{1})=\varnothing$, as $\mathbf{1}\in S$. Setting $E:=A\cup (A+\mathbf{1})$, Lemma \ref{lemAdditivity} implies $d^{*}(E)> 1-2\varepsilon$, since $A\cap (A+\mathbf{1})=\varnothing$ and $d^{*}(A)>\frac{1}{2}-\varepsilon$.  Furthermore $E+S = (A+S) \cup (A+S+\mathbf{1})$.  Since $A+S$ is not piecewise syndetic, $E+S$ is also not piecewise syndetic, by the translation invariance and partition regularity of piecewise syndeticity (Lemma \ref{lemmaPWPartitionRegular}).  Thus to prove Theorem \ref{theoremPWsynd} it suffices to establish Equation (\ref{equationDiff2}).

To prove Equation (\ref{equationDiff1}), observe that Lemma \ref{lemmaPush} implies
\[
A+S \subseteq A':= \bigcup_{l=1}^{\infty} A_{0}((n_{1}, m_{1}-k_{1}),\dots, (n_{l},m_{l}-k_{l})).
\]  Using the fact that $m_{j} = 3k_{j}$ for each $j$, we find that $A' \cap A=\varnothing$, by Part (iv) of Lemma \ref{lemmaProperties} and Lemma \ref{lemmaMonotone}. We then have $(A+S)\cap A = \varnothing$, so $(A-A)\cap S = \varnothing$, completing the proof of Equation (\ref{equationDiff1}).  The same argument also shows that $(A'-A')\cap S =\varnothing$, establishing Equation (\ref{equationDiff2}).
\end{proof}

\begin{proof}[Proof of Theorems \ref{theoremMainDynamic} and \ref{theoremPreciseDynamic}.]
 In the proof of Theorem \ref{theoremMainDensity} we constructed $S, A\subseteq G_{2}$ such that $d^{*}(A)>\frac{1}{2}-\varepsilon$, every translate of $S$ is a set of chromatic recurrence (and therefore a set of topological recurrence), and $(A-A)\cap S=\varnothing$.  Theorem \ref{theoremPreciseDynamic} follows from the existence of these sets and Lemma \ref{lemmaMeasureCorrespondence}.    Theorem \ref{theoremMainDynamic} follows immediately from Theorem \ref{theoremPreciseDynamic}.
\end{proof}

\section{Appendix}\label{sectionAppendix}

This section contains material needed for the exposition in Section \ref{sectionQuestions}.

\subsection{Equivalent forms}

For the reader's convenience, we provide well known equivalent formulations of the properties $(R_{j})$ and $(R_{j}^{\bullet})$, for $j=4,5, 6$, defined in Sections \ref{sectionDefinitions} and \ref{sectionTranslations}.

We need the following definition to state the next lemma.

\begin{definition}\label{defSpecial}
  Let $G$ be a countable abelian group.  We say that $S\subseteq G$ is a \emph{set of measurable recurrence for group rotation $G$-systems} if for every such $G$-system $(Z,m,R_{\rho})$ and every $D\subseteq Z$ having $m(D)>0$, there exists $g\in G$ such that $m(D\cap R_{\rho}^{g}D)>0$.
\end{definition}

In the following lemma $\overline{E}$ denotes the topological closure of a subset $E$ of a compact abelian group.

\begin{lemma}\label{lemmaBohr}
  Let $G$ be a countable abelian group and $S\subseteq G$.  The following are equivalent.
  \begin{enumerate}
    \item[(i)]  $S$ satisfies $(R_{6})$.
    \item[(ii)]  If $W\subseteq G$ is a Bohr neighborhood of $0$, then $S\cap W\neq \varnothing$.
    \item[(iii)]  For every homomorphism $\rho: G\to Z$, where $Z$ is a compact abelian group, $\overline{\rho(S)}$ contains $0\in Z$.
    \item[(iv)]  $S$ is a set of measurable recurrence for group rotation $G$-systems.
     \end{enumerate}
\end{lemma}

\begin{proof}
We prove (iii) $\implies$ (ii) $\implies$ (i) $\implies$ (iii), then (iv) $\implies$ (i) and (iii) $\implies$ (iv).

  (iii) $\implies$ (ii).  Suppose $S$ satisfies condition (iii), and let $W\subseteq G$ be a Bohr neighborhood of $0$, so that  there are homomorphisms $\rho_{1},\dots,\rho_{k}: G\to \mathbb T$ and a neighborhood $V$ of $0\in \mathbb T$ such that $W$ contains $\bigcap_{i=1}^{k} \rho_{i}^{-1}(V)$. Let $Z$ be the group $\mathbb T^{k}$ and $\rho: G\to Z$ be the homomorphism given by $\rho(g)= (\rho_{1}(g),\dots,\rho_{k}(g))$.    Let $U=V\times~\cdots~\times~V\subseteq Z$, so that $U$ is a neighborhood of $0\in Z$.  Since $S$ satisfies condition (iii), there is a $g\in S$ such that $\rho(g)\in U$.  Then $\rho_{i}(g)\in V$ for each $i$, and we conclude that $g\in W$.

  (ii) $\implies$ (i).  Suppose $S\cap W\neq \varnothing$ for every Bohr neighborhood of $0\in G$.  Let $(Z,R_{\rho})$ be a minimal group rotation, and let $U\subseteq Z$ be a nonempty open set.  Choose a neighborhood $V$ of $0\in Z$ such that $U\cap (U+v)\neq \varnothing$ for every $v\in V$, and let $W=\rho^{-1}(V)$.  Then $W$ is a Bohr neighborhood\footnote{Here we are using the fact that every homomorphism from $G$ to a compact group is continuous in the Bohr topology.} of $0$, so there exists $g\in S\cap W$, meaning $\rho(g) \in V$.  We then have $U\cap (U+\rho(g)) \neq \varnothing$, meaning $U\cap (R_{\rho}^{g}U)\neq \varnothing$.  This shows that $S$ satisfies $(R_{6})$.

  (i) $\implies$ (iii).  Let $\rho: G\to Z$ be a homomorphism to a compact abelian group $Z$ and assume $S\subseteq G$ satisfies $(R_{6})$.    Assume, without loss of generality, that $\overline{\rho(G)} = Z$.  Let $U\subseteq Z$ be a neighborhood of $0\in Z$, and let $V\subseteq U$ be such that $V-V\subseteq U$.  Since $S$ satisfies $R_{6}$, there is a $g\in S$ such that $V\cap (V+\rho(g)) \neq \varnothing$, meaning $\rho(g)\in V-V$, so $\rho(g)\in U$.  Since $U$ is an arbitrary neighborhood of $0\in Z$, we have shown that $0\in \overline{\rho(S)}$.

(iv) $\implies$ (i).  Suppose $S$ satisfies condition (iv) and that $(Z,R_{\rho})$ is a minimal group rotation.  Let $U\subseteq Z$ be a nonempty open set, so that $m(U)>0$, where $m$ is the Haar probability measure on $Z$.  Condition (iv) implies $m(U\cap (U+\rho(g)))>0$ for some $g\in S$, so that $U\cap (U+\rho(g))\neq \varnothing$ for this $g$.  Since $U\subseteq Z$ was an arbitrary open set, this shows that $S$ satisfies $(R_{6})$.

(iii) $\implies$ (iv).  Let $(Z,m,R_{\rho})$ be a group rotation $G$-system.  For  measurable sets $D\subseteq Z$, the map $z\mapsto m(D\cap (D+z))$ is continuous.  If $m(D)>0$, we conclude that $m(D\cap (D+z))>0$ for all $z$ sufficiently close to $0$, so condition (iii) implies $m(D\cap (D+\rho(g)))>0$ for some $g\in S$.  \end{proof}

See Section \ref{sectionCombinatorial} for the definitions of ``measure expanding," ``measure transitive," etc.

\begin{lemma}\label{lemmaExpanding}
Let $G$ be a countable abelian group and $S\subseteq G$.  The following are equivalent.

\begin{enumerate}
  \item[(i)]  $S$ satisfies $(R_{6}^{\bullet})$.
  \item[(ii)] $S$  is dense in the Bohr topology of $G$.
  \item[(iii)] Every translate of $S$ is a set of measurable recurrence for group rotation $G$-systems.
\end{enumerate}

The following are equivalent.

\begin{enumerate}
\item[(iv)] Every translate of $S$ is a set of chromatic recurrence.
  \item[(v)] $S$ satisfies $(R_{5}^{\bullet})$.
  \item[(vi)] $S$ is transitive for minimal $G$-systems.
  \item[(vii)] $S$ is expanding for minimal $G$-systems. \end{enumerate}

The following are equivalent.

\begin{enumerate}
  \item[(viii)] Every translate of $S$ is a set of density recurrence.
  \item[(ix)] $S$ satisfies $(R_{4}^{\bullet})$.
  \item[(x)]  $S$ is measure transitive.
  \item[(xi)]  $S$ is measure expanding.
\end{enumerate}

\end{lemma}

\begin{proof}
 The equivalence of (i), (ii), and (iii) is a consequence of Lemma \ref{lemmaBohr} and the definition of the Bohr topology.

%  (ii) $\implies$ (iii)  Let $S$ be dense in the Bohr topology of $G$.  The hypothesis on $S$ is translation invariant, so  it suffices to show that $S$ is a set of measurable recurrence for minimal group rotations. Let $(Z,m,R_{\rho})$ be a minimal group rotation, so that $\rho(G)$ is dense in $Z$.  Let $D\subseteq Z$ have $m(D)>0$.  Using regularity of Haar measure, choose a neighborhood $U$ of $0\in Z$ such that $m(D\cap (D+z))>0$ whenever $z\in U$.  Since $S$ is dense in the Bohr topology of $G$ and $\rho(G)$ is dense in $Z$, we have $\rho(S)\cap U\neq \varnothing$, meaning $m(D\cap (R_{\rho}^{g}D)))>0$ for some $g\in S$.
%
%  (iii) $\implies$ (i)  The hypothesis on $S$ is translation invariant, so it suffices to prove that $S$ satisfies $(R_{6})$.  Let $Z$ be a compact abelian group with Haar measure $m$. If $(Z,R_{\rho})$ is a minimal group rotation and $U\subseteq Z$ is a nonempty open set, then $m(U)>0$.  Since $S$ is a set of measurable recurrence for group rotation $G$-systems  there exists $h\in S$ such that $m(U\cap (R_{\rho}^{h} U))>0$, so $U\cap (R_{\rho}^{h}U)\neq \varnothing$.  This proves that $S$ satisfies $(R_{6})$.

The equivalence of (iv) and (v) is due to Part (ii) of Proposition \ref{propositionCorrespondence}.

(v) $\implies$ (vi).  Let $(X,T)$ be a minimal topological $G$-system with $U, V \subseteq X$ nonempty open sets.  By the minimality of $(X,T)$ there exists $g\in G$ such that $W:=U\cap T^{g} V\neq \varnothing$.  Since $S-g$ is a set of topological recurrence, there exists $h\in S$ such that $W\cap T^{h-g}W\neq \varnothing$, meaning
\[
U\cap T^{g}V\cap T^{h-g}(U\cap T^{g}V) \neq \varnothing,
\]
which implies $U\cap T^{h} V \neq \varnothing$.  Since $U, V\subseteq X$ are arbitrary nonempty open sets, we have shown that $S$ is transitive for minimal $G$-systems.

(vi) $\implies$ (vii).  Let $(X,T)$ be a minimal topological $G$-system and let $x\in X$.  We will show that for each nonempty open set $U$, there is a dense open set of points $V_{U}$ such that for all $x\in V_{U}$, there exists $g\in S$ such that $T^{g}x\in U$.  Fix a nonempty open set $U\subseteq X$.  By condition (vi) for every nonempty open $W\subseteq X$ there exists $g\in S$ such that $T^{g}W\cap U \neq \varnothing$.  For each open set $W$ and  each $g\in S$ let $W_{g}=W\cap T^{-g}U$.  Then $V_{U}:=\bigcup_{W\subseteq X \text{ open},\, g\in G} W_{g}$ is the desired dense open set.

Let $\mathcal U$ be a countable base for the topology of $X$.  Then $E:=\bigcap_{U\in \mathcal U} V_{U}$ is a dense $G_{\delta}$ set such that for all $x\in E$, $\{T^{g}x: g\in S\}$ is dense in $X$.

(vii) $\implies$ (iv).  Condition (vii) is translation invariant, so it suffices to prove that if (vii) holds, then $S$ is a set of topological recurrence.  Let $(X,T)$ be a minimal topological $G$-system and $U\subseteq X$ a nonempty open set.  Condition (vii) permits us to choose $x\in U$ and $g\in S$ such that $T^{g}x\in U$.  We conclude that $U\cap T^{g}U\neq \varnothing$ for this $g$.  Since $U$ is an arbitrary nonempty open set, we conclude that $S$ is a set of topological recurrence.

The equivalence of (viii) and (ix) follows directly from Part (i) of Proposition \ref{propositionCorrespondence}.

(ix) $\implies$ (x).  Let $S\subseteq G$ satisfy (ix) and let $(X,\mu,T)$ be a measure preserving $G$-system.  Let $C, D \subseteq X$ and $g\in G$ be such that $\mu(C \cap T^{g}D)>0$.  Let $E= C\cap T^{g}D$.  Since $S-g$ is a set of measurable recurrence, there is an $h\in S$ such that $\mu(E\cap T^{h-g}E)>0$.  Consequently
\[
\mu(C\cap T^{g}D \cap T^{h-g}(C\cap T^{g}D)) >0,
\]
which implies $\mu(C\cap T^{h}D)>0$.  Hence $S$ is measure transitive.

(x) $\implies$ (xi).  Let $S\subseteq G$ satisfy condition (x) and let $(X,\mu,T)$ be an ergodic measure preserving $G$-system.  Let $D\subseteq X$ have $\mu(D)>0$, and let $E=\bigcup_{g\in S}T^{g}D$. Suppose, to get a contradiction, that $\mu(X~\setminus~E)>0$. By the ergodicity of $(X,\mu,T)$ there exists $g\in S$ such that $\mu((X\setminus E)\cap T^{g}D)>0$.  Condition (x) then implies there exists $h\in S$ such that $\mu((X\setminus E) \cap T^{h}D)>0$, a contradiction, since $T^{h}D\subseteq E$.

(xi) $\implies$ (viii).  Since condition (xi) is translation invariant, it suffices to prove that $S$ is a set of density recurrence.  Let $A\subseteq G$ have $d^{*}(A)>0$.  Then by Lemma \ref{lemmaMeasureCorrespondence} there is an ergodic measure preserving system $(X,\mu,T)$ and a set $D\subseteq X$ having $\mu(D)>0$ such that $A-A$ contains $\{g\in G: \mu(D\cap T^{g}D)>0\}$.  Now condition (xi) implies that there exists $g\in S$ such that $\mu(D\cap T^{g}D)>0$ (since otherwise $\mu\bigl(\bigcup_{g\in S} T^{g}D\bigr)\leq1-\mu(D)$) and we conclude that $(A-A)\cap S\neq \varnothing$.   Since $A\subseteq G$ was an arbitrary set having $d^{*}(A)>0$, we have shown that $S$ is a set of density recurrence. \end{proof}

\subsection{Minimality and shift spaces}  Let $G$ be a countable abelian group. See \cite[Theorem 1.15]{Furstenberg} for a proof of the following standard lemma.

\begin{lemma}\label{lemmaMinimalSyndetic}
  If $(X,T)$ is a minimal topological $G$-system and $U\subseteq X$ is a nonempty open set, then for all $x\in X$ the set $\{g: T^{g}x\in U\}$ is syndetic.
\end{lemma}

Consider the compact metrizable space $X=\{0,1\}^{G}$.  The elements of $G$ are functions $x:G\to \{0,1\}$.  Let $\sigma$ be the action of $g$ on $X$ defined by $(\sigma^{g}x)(h)=x(h+g)$ for each $g, h\in G$, $x\in X$.  We call the topological $G$-system $(X,\sigma)$ the \emph{shift space} and $\sigma$ the \emph{shift action}. Note that a sequence of elements $x_{n}$ converges to $x\in X$ if and only if for every finite $F\subseteq G$, $x_{n}|_{F} = x|_{F}$ for all sufficiently large $n$, meaning $X$ has the topology of pointwise convergence.

Note that for $A\subset G$, $1_{A}\in X$ and for $g\in G$, $\sigma^{g}1_{A}=1_{A-g}$.

For the next lemma, recall that a set $A\subseteq G$ is thick if for every finite $F\subseteq G$, there is a $g\in G$ such that $F+g\subseteq A$ (Definition \ref{definitionSyndetic}).

\begin{lemma}\label{lemmaPWorbits}
Let $G$ be a countable abelian group and let $(X,\sigma)$ be the corresponding shift space.  Let $A\subseteq G$. Consider the following conditions.

\begin{enumerate}
  \item[(i)]  $A+F$ is thick for some finite set $F\subseteq G$.

  \item[(ii)]  The orbit closure $\overline{\{\sigma^{g}1_{A}: g\in G\}}$ contains a minimal subsystem not equal to $\{x_{0}\}$, where $x_{0} \in \{0,1\}^{G}$ is the constant $0$ function.
\end{enumerate}
Then \textup{(i) $\implies$ (ii)}.

\end{lemma}

\begin{proof} Let $F\subseteq G$ be a finite set such that $F+A$ is thick.  Let $E_{1}\subseteq E_{2}\subseteq \dots$ be an increasing sequence of finite sets whose union is $G$.  For each $i\in \mathbb N$, let $g_{i}$ satisfy $E_{i}+g_{i}\subseteq F+A$.  Consider the points $\sigma^{g_{i}}1_{A}= 1_{A-g_{i}}$, and let $x\in X$ be a limit of these points.

  \begin{claim}
    For all $g\in G$, there exists $h\in F$ such that $x(g-h)=1$.
  \end{claim}
\begin{proof}[Proof of Claim]
 Let $g\in G$, and choose $N$ so that $g\in E_{i}$ for all $i\geq N$.  For each $i\geq N$, we have that $A-g_{i}+F\supset E_{i}\ni g$, meaning
 \[\sum_{h\in F}1_{A-g_{i}}(g-h)\geq 1.\]  Then \begin{equation}\label{equationFinite}
   1\leq \liminf_{i\to \infty} \sum_{h\in F} 1_{A-g_{i}}(g-h)\leq  \sum_{h\in F} x(g-h),
 \end{equation} so $x(g-h)=1$ for some $h\in F$.
\end{proof}
   The Claim implies that $x_{0}$ is not in the orbit closure $Y$ of $x$. It follows that no minimal subsystem of $X$ is contained in the orbit closure of $x$ is equal to $\{x_{0}\}$.  Since $x$ is in the orbit closure of $1_{A}$, we can consider any minimal subsystem contained in $Y$ to establish condition (ii).
\end{proof}

\subsection{Piecewise syndeticity}

Condition (i) in the following lemma is the standard definition of ``piecewise syndetic," condition (iii) is our definition.

\begin{lemma}\label{lemmaPWSequivalents}
  Let $G$ be a countable abelian group and $A\subseteq G$.  The following are equivalent.

  \begin{enumerate}
    \item[(i)] There is a syndetic set $S\subseteq G$ and a thick set $R\subseteq G$ such that $S\cap R\subseteq A$.
    \item[(ii)] There is a finite set $F$ such that $A+F$ is thick.
    \item[(iii)] There is a syndetic set $S\subseteq G$ such that for all finite $F\subseteq S$, there is a $g\in G$ such that $F+g\subseteq S$.  \end{enumerate}
\end{lemma}

\begin{proof}
(i) $\implies$ (ii)    Let $S, R\subseteq G$ be syndetic and thick sets, respectively, such that $S\cap R\subseteq A$.   Let $F\subseteq G$ be a finite set such that $S+F=G$ and let $K$ be an arbitrary finite subset of $S$.

Choose $g$ so that $K-F+g\subseteq R$.  Since $S+F=G$, we may choose elements $s_{1},\dots, s_{n}\in S$ and $f_{1},\dots, f_{n}\in F$ so that $K+g=\{s_{i}+f_{i}: 1\leq i \leq n\}$.  Then for each $i\leq n$, $s_{i}\in K-F+g$, so $s_{i}\in S\cap R$.  It follows that $K+g\subseteq (S\cap R)+F$, so $A$ satisfies condition (ii).

(ii) $\implies$ (iii)  Suppose $F\subseteq G$ is finite and $A+F$ is thick.  Let $X=\{0,1\}^{G}$ be the shift space with shift action $\sigma$, and let $1_{A}\in X$ be the characteristic function of $A$.  By Lemma \ref{lemmaPWorbits} there is a minimal subsystem $Y$ contained in the orbit closure of $1_{A}$ which is not equal to $\{x_{0}\}$, where $x_{0}$ is the constant $0$ function.  If $y\in Y$, Lemma \ref{lemmaMinimalSyndetic} implies the set $S:=\{g: y(g)=1\}$ is syndetic.  Since $y$ is in the orbit closure of $1_{A}$, we have that for all finite $E\subseteq G$, there exists $g_{E}\in G$ such that $1_{A}(g)=y(g+g_{E})$ for all $g\in E$.  It follows that $A$ contains $(S\cap E)-g_{E}$ for each finite $E\subseteq G$, so condition (iii) is satisfied.

The implication (iii) $\implies$ (ii) is straightforward.

To prove (ii) $\implies$ (i) we need the following definitions.

Let $(X,T)$ be a topological $G$-system and $d$ a metric on $X$ generating the topology on $X$.  We say that $x, y\in X$ are \emph{proximal} if there is a sequence of elements $g_{n}\in G$ such that $d(T^{g_{n}}x,T^{g_{n}}y)\to 0$ as $n\to \infty$.

If $A\subseteq X$ and $x\in X$, we say that $x$ is \emph{proximal to $A$} if there is a sequence of elements $g_{n}\in G$ such that $\inf_{a\in A} d(T^{g_{n}}x,T^{g_{n}}a)\to 0$ as $n\to \infty$.

Proposition 8.6 of \cite{Furstenberg} says that if $A$ is a $T$-invariant closed subset of $X$ and $x\in X$ is proximal to $A$, then $x$ is proximal to some $y\in A$.

In the shift space $(X,\sigma)$, proximality of $x$ and $y$ is equivalent to the following condition:
\begin{align}\label{equationWindow}  \text{For all finite $F\subseteq G$, there is a $g\in G$ such that $x|_{F+g}=y|_{F+g}$.}
\end{align}
We now prove (ii) $\implies$ (i).  Let $A, F\subseteq G$ be such that $F$ is finite and $A+F$ is thick.  Let $(X,\sigma)$ be the shift space, and by Lemma \ref{lemmaPWorbits} let $Y\subseteq X$ be a minimal subsystem of the orbit closure of $1_{A}$ such that $x_{0}\notin Y$.  By Proposition 8.6 of \cite{Furstenberg} we may choose $y\in Y$ such that $1_{A}$ is proximal to $y$.  The minimality of $Y$ implies $S:=\{g: y(g)=1\}$ is syndetic, and condition (\ref{equationWindow}) then implies $A$ satisfies condition (i). \end{proof}

\begin{lemma}\label{lemmaPWPartitionRegular}  Let $G$ be a countable abelian group.  If $G=\bigcup_{i=1}^{k} A_{i}$, then there is an $i\leq k$
 such that $A_{i}$ is piecewise syndetic.  Consequently, if $A, B\subseteq G$ are not piecewise syndetic, then $A\cup B$ is not piecewise syndetic.
\end{lemma}

For a proof, see Theorem 4.40 of \cite{HindmanStrauss}.

\subsection{Correspondence Principle}
We need the following lemma to relate the properties $(R_{i})$ and $(R_{i}^{\bullet})$ to properties of difference sets and state various forms of the conditions $(R_{i})$ defined in Section \ref{sectionDefinitions}.

\begin{lemma}\label{lemmaMeasureCorrespondence}
  Let $G$ be a countable abelian group and let $A\subseteq G$ have $d^{*}(A)>0$.  Then there is an ergodic measure preserving $G$-system $(X,\mu,T)$ and a set $D\subseteq X$ with $\mu(D)\geq d^{*}(A)$ such that $A-A$ contains $\{g: \mu(D\cap T^{g}D)>0\}$.
\end{lemma}
Lemma \ref{lemmaMeasureCorrespondence} is proved for $G=\mathbb Z$ in Theorem 3.18 of \cite{Furstenberg}.  Ergodicity is not mentioned there, but the proof is easily modified to obtain it.  For an outline of a proof in the general case, see \cite[Lemma 5.1]{Griesmer3Fold}.

\subsection{Implications}\label{sectionImplications} We prove $(R_{1}) \implies (R_{2})$ and briefly discuss the implications $(R_{i})\implies (R_{i+1})$ for $i\geq 2$.

We need some tools from harmonic analysis on compact abelian groups, as presented in \cite{RudinFourier}.  If $G$ is a countable abelian group, equip $G$ with the discrete topology, and let $\widehat{G}$ denote the group of homomorphisms (or \emph{characters}) $\chi:G\to \mathcal S^{1}\subseteq \mathbb C$ with the topology of pointwise convergence and the group operation of pointwise multiplication.  Then $\widehat{G}$ is a compact abelian group.  We write $\chi_{0}$ for the identity element of $\widehat{G}$, which is the constant character: $\chi_{0}(g)=1$ for all $g\in G$.

\begin{lemma}
  $(R_{1}) \implies (R_{2})$.
\end{lemma}

\begin{proof}  Let $S$ satisfy ($R_{1}$), so that we may choose a sequence $(S_{j})_{j\in \mathbb N}$ of finite sets $S_{j}\subseteq S$, satisfying
\begin{align}\label{equationEquidsitribution}
  \lim_{j\to \infty} \frac{1}{|S_{j}|}\sum_{g\in S_{j}} \chi(g) =0 && \text{ for all } \chi\in \widehat{G}\setminus\{\chi_{0}\}.
\end{align}

  Let $(X,\mu,T)$ be a measure preserving $G$-system and $D\subseteq X$ a measurable set.  We will prove that
\begin{align}\label{equationErgodicInequality}
  \lim_{j\to \infty} \frac{1}{|S_{j}|}\sum_{g\in S_{j}}\mu(D\cap T^{g}D)\geq \mu(D)^{2},
\end{align}
which implies that for all $\varepsilon>0$ there exists $j\in \mathbb N$ and $g\in S_{j}$ such that $\mu(D\cap T^{g}D)>\mu(D)^{2}-\varepsilon$.  It therefore suffices to prove Inequality (\ref{equationErgodicInequality}) to show that $S$ satisfies ($R_{2}$).

   Let $f=1_{D}$.  Then $f\in L^{2}(\mu)$, and $\mu(D\cap T^{g}D) = \int f\cdot f\circ T^{g} \,d\mu$.  The action $U_{T}$ of $G$ on $L^{2}(\mu)$ given by $U_{T}^{g} f = f\circ T^{g}$ is unitary, meaning that for each $g$, $U_{T}^{g}:L^{2}(\mu)\to L^{2}(\mu)$ is an invertible linear isometry.  The Bochner-Herglotz theorem therefore implies the existence of a positive Borel measure $\sigma$ on $\widehat{G}$ such that $\int  f\cdot f\circ T^{g} \,d\mu = \int \chi(g) \,d\sigma(\chi)$ for all $g\in G$.  We have $\sigma(\widehat{G}) =  \int 1_{\widehat{G}} \,d\sigma(\chi) = \int \chi(0)\, d\sigma(\chi) = \mu(D)$. Then
\begin{align}
\label{equationAverageLimit} \begin{split}
    \lim_{j\to \infty} \frac{1}{|S_{j}|}\sum_{g\in S_{j}} \mu(D\cap T^{g}D) &=\lim_{j\to \infty} \frac{1}{|S_{j}|}\sum_{g\in S_{j}}\int f\cdot f\circ T^{g}\,d\mu\\
  &= \lim_{j\to \infty} \frac{1}{|S_{j}|}\sum_{g\in S_{j}}\int \chi(g)\, d\sigma(\chi)\\
  &= \int \lim_{j\to \infty} \frac{1}{|S_{j}|}\sum_{g\in S_{j}} \chi(g)\, d\sigma(\chi)\\
  &= \int 1_{\{\chi_{0}\}}(\chi) \,d\sigma(\chi)\\
  &= \sigma(\{\chi_{0}\}).
\end{split}
\end{align}
The limit in (\ref{equationAverageLimit}) is therefore independent of the sequence $(S_{j})_{j\in \mathbb N}$, as long as Equation (\ref{equationEquidsitribution}) is satisfied.  When $\mathbf{\Phi}=(\Phi_{j})_{j\in \mathbb N}$ is a F{\o}lner sequence,  $\mathbf{\Phi}$ satisfies (\ref{equationEquidsitribution}) and the mean ergodic theorem (see \cite{Glasner}) implies that $\lim_{j\to \infty} \frac{1}{|\Phi_{j}|}\sum_{g\in \Phi_{j}} f\circ T^{g} = Pf$, where $Pf$ is the orthogonal projection of $f$ onto the closed space of $T$-invariant functions in $L^{2}(\mu)$.  By (\ref{equationAverageLimit}), we then have
\begin{align*}
  \sigma(\{\chi_{0}\}) &= \int f\cdot Pf \,d\mu\\
  &= \int Pf\cdot Pf \,d\mu && \text{since $P$ is an orthogonal projection}\\
  &\geq \Bigl(\int Pf \,d\mu\Bigr)^{2} && t\mapsto t^{2}\text{ is convex}\\
  &= \Bigl( \int f \,d\mu\Bigr)^{2} && P1_{X} = 1_{X}\\
  &= \mu(D)^{2} && f=1_{D},
\end{align*}
so $\sigma(\{\chi_{0}\})\geq \mu(D)^{2}$, and (\ref{equationAverageLimit}) then implies Inequality (\ref{equationErgodicInequality}). \end{proof}

The implications $(R_{2}) \implies (R_{3})$ and $(R_{3})\implies (R_{4})$ are straightforward.  The implication $(R_{4})\implies (R_{5})$ is a consequence of the Bogoliouboff-Kryloff Theorem: every topological $G$-system $(X,T)$ admits a $T$-invariant probability measure.  Consequently, every minimal topological $G$-system $(X,T)$ admits a $T$-invariant probability measure $\mu$ having full support, since the support of a $T$-invariant measure is a $T$-invariant compact subset of $X$.  It follows that every nonempty open set $U\subseteq X$ has $\mu(U)>0$, and then the fact that $S$ satisfies $(R_{4})$ implies that there is a $g\in S$ such that $U\cap T^{g}U\neq \varnothing$.

The implication $(R_{5}) \implies (R_{6})$ is straightforward.

  \bibliographystyle{amsplain}
\frenchspacing
\bibliography{Recurrence_problems}

\end{document}